\documentclass[11pt]{amsart}

\usepackage{amsmath,amssymb,mathrsfs}
\usepackage{microtype}
\usepackage[margin=0.65in]{geometry}
\usepackage{tikz}
\usetikzlibrary{arrows.meta,calc,decorations.pathreplacing}

\definecolor{ChannelBlue}{RGB}{30,104,142}
\definecolor{CapRed}{RGB}{184,72,65}
\definecolor{SmoothOrange}{RGB}{222,139,42}
\definecolor{DomainFill}{RGB}{232,244,248}
\definecolor{GuideGray}{RGB}{105,112,118}

\usepackage[hidelinks]{hyperref}
\hypersetup{
 pdftitle={Nonconvex Sublevel Sets for the Planar Translating Mean Curvature Equation},
 pdfauthor={Jiahuan Li, Yilu Liu, Xi-Nan Ma, Guohuan Qiu}
}

\allowdisplaybreaks
\numberwithin{equation}{section}

\newtheorem{theorem}{Theorem}[section]
\newtheorem{lemma}[theorem]{Lemma}
\newtheorem{proposition}[theorem]{Proposition}
\newtheorem{corollary}[theorem]{Corollary}
\theoremstyle{remark}

\newcommand{\R}{\mathbb R}
\newcommand{\Qop}{\mathcal Q}
\newcommand{\Fop}{\mathcal F}
\newcommand{\Acal}{\mathcal A}
\newcommand{\Wcal}{\mathcal W}
\newcommand{\Om}{\Omega}
\newcommand{\dd}{\,\mathrm d}

\title[A PLANAR TRANSLATOR COUNTEREXAMPLE]
{NONCONVEX SUBLEVEL SETS FOR THE PLANAR\\
TRANSLATING MEAN CURVATURE EQUATION}

\author{Jiahuan Li}
\address{School of Mathematical Sciences, University of Science and
Technology of China, Hefei, 230026, Anhui Province, P.R. China}
\email{jiahuan@mail.ustc.edu.cn}

\author{Yilu Liu}
\address{School of Mathematical Sciences, University of Science and
Technology of China, Hefei, 230026, Anhui Province, P.R. China}
\email{liuylgeoanaly@mail.ustc.edu.cn}

\author{Xi-Nan Ma}
\address{School of Mathematical Sciences, University of Science and
Technology of China, Hefei, 230026, Anhui Province, P.R. China}
\email{xinan@ustc.edu.cn}

\author{GUOHUAN QIU}

\address{Institute of Mathematics, Academy of Mathematics and Systems Science,
Chinese Academy of Sciences, No. 55 Zhongguancun East Road, Beijing 100190,
China}
\email{qiugh@amss.ac.cn}


\begin{document}

\begin{abstract}
Translating solitons arise as models for type~II singularities of
mean-convex mean curvature flow.  We construct a smooth bounded uniformly
convex domain \(\Om\Subset\R^2\) such that the zero-Dirichlet solution of
the planar translating mean curvature equation has a nonconvex sublevel
set.  The construction is based on a corrected
near-critical grim-reaper profile and explicit barriers on a long convex
channel.
\end{abstract}

\maketitle

\section{Introduction}

Singularity formation is a central issue in mean curvature flow, and the
geometry of blow-up limits governs the flow near a singular time.  Let
\(\{M_t\}_{0\leq t<T}\) be a smooth compact mean
curvature flow that becomes singular at the finite time \(T\), and let
\(A\) denote the second fundamental form of \(M_t\).  The singularity is
of type~I if
\[
 \sup_{0\leq t<T}(T-t)\sup_{M_t}|A|^2<\infty,
\]
and of type~II otherwise.  Thus type~I curvature blow-up is at most of
order \((T-t)^{-1/2}\), whereas at a type~II singularity the curvature
concentrates faster along some sequence of space-time points.  Parabolic
blow-ups at type~I singularities are modeled by self-shrinking solutions.
For mean-convex flow, suitable blow-ups at type~II singularities give
convex translating solutions \cite[pp.~1185--1186]{Wang2011}.
Accordingly, translating solutions play the same role in the analysis of
type~II singularities as self-shrinking solutions do for type~I
singularities.

This connection makes the classification and level-set geometry of
translators part of the analysis of singularity models.  Wang formulated
a Bernstein-type classification problem for complete translators and a
related propagation problem for convex level sets; the latter was noted
to remain open even for the constant prescribed mean curvature equation
\cite[Section~5]{WangICCM2008}.  Although the present paper concerns a
bounded-domain problem rather than a complete translator, solutions of
bounded Dirichlet problems enter this classification program as
approximants in constructions of complete solutions
\cite[p.~1189]{Wang2011}.

Let \(\Om\Subset\R^2\) be a smooth bounded domain.  We consider the
Dirichlet problem
\begin{equation}\label{eq:intro-problem}
 \begin{cases}
 \displaystyle
 \operatorname{div}\!\left(\dfrac{Du}{\sqrt{1+|Du|^2}}\right)
 =\dfrac{1}{\sqrt{1+|Du|^2}}
 &\text{in }\Om,\\[1.1ex]
 u=0&\text{on }\partial\Om.
 \end{cases}
\end{equation}
The graph of a solution translates by mean curvature flow in a fixed
vertical direction.  On a smooth bounded mean convex domain,
\eqref{eq:intro-problem} has a unique classical solution
\cite[Theorem~C.2]{Zhou2024}.  The maximum principle gives \(u<0\) in
\(\Om\).  Thus the relevant convexity question is whether the sublevel
sets
\[
 \Om_c=\{x\in\Om:u(x)<c\},\qquad \min_{\Om}u<c<0,
\]
inherit the convexity of the domain.
The graphical translating mean curvature flow with Dirichlet data,
including its long-time existence and convergence, was studied by Ma
\cite{Ma2018}.

For the level-set-flow equation, Wang's compactness argument for bounded
Dirichlet approximations uses the concavity of \(\log(-u)\).  Wang
explicitly observed that the analogous logarithmic concavity was still
open for the translating mean curvature equation
\cite[p.~1189]{Wang2011}.  In the present notation, the question is
whether \(\log(-u)\) must be concave when \(\Om\) is convex.  Since every
superlevel set of a concave function is convex, an affirmative answer
would force every sublevel set \(\{u<c\}\), \(c<0\), to be convex.

From the PDE viewpoint, convexity of level sets is a classical theme in
elliptic equations.
Concavity maximum principles and power-concavity methods yield positive
results for a number of semilinear and quasilinear Dirichlet problems;
we refer to
\cite{CaffarelliSpruck1982,Korevaar1983,Kennington1985,Kawohl1985}
and the references therein.  For equations of mean-curvature type,
Shiffman proved convexity of the horizontal sections of a minimal
annulus bounded by convex curves in parallel planes
\cite{Shiffman1956}; Korevaar obtained the corresponding graph result
in higher dimensions \cite{Korevaar1990}.  For the prescribed constant
mean curvature equation over a convex planar domain, Sakaguchi proved a
positive result in the small-curvature regime \cite{Sakaguchi1989}.

Convexity of the domain alone, however, does not force quasiconvexity.
Hamel, Nadirashvili, and Sire constructed counterexamples for
semilinear equations in convex domains and convex rings
\cite{HamelNadirashviliSire2016}.  Wang disproved the general
level-set convexity conjecture for the constant mean curvature equation,
even on a smooth uniformly convex domain \cite{Wang2014}.  More
recently, Zhang showed that minimal strictly stable solutions of
Gelfand and shifted-power problems may have nonconvex superlevel sets
in smooth uniformly convex planar domains \cite{Zhang2026}.
For complete \(2\)-convex translators without boundary, Spruck and Sun
proved a convexity theorem \cite{SpruckSun2021}.  In contrast, the main
result below shows that convexity can fail for bounded zero-Dirichlet
solutions of the translating mean curvature equation.

\begin{theorem}\label{thm:main}
There exists a bounded \(C^\infty\) uniformly convex domain
\(\Om\Subset\R^2\) such that the unique solution
\(u\in C^\infty(\Om)\cap C(\overline\Om)\) of
\eqref{eq:intro-problem} has a nonconvex sublevel set.  More precisely,
\[
 \{x\in\Om:u(x)<-\log 2\}
\]
is not convex. 
\end{theorem}

In particular, Theorem~\ref{thm:main} gives a negative answer to the
logarithmic concavity question from \cite{Wang2011} in the bounded
zero-Dirichlet setting, already in dimension two and on a smooth uniformly
convex domain.  Indeed, for every \(c<0\),
\[
 \{u<c\}=\{\log(-u)>\log(-c)\}.
\]
If \(\log(-u)\) were concave, the right-hand side would be convex.  Taking
\(c=-\log 2\) contradicts Theorem~\ref{thm:main}.  This conclusion concerns
bounded Dirichlet solutions; it does not resolve the distinct
entire-solution propagation problem formulated in
\cite[Section~5]{WangICCM2008}.

The proof uses a corrected grim-reaper profile on a long channel.  A
direct slow-variable stretching, in the spirit of Wang's construction,
first leads to a one-dimensional profile.  For
\eqref{eq:intro-problem} that profile is necessarily a grim reaper.
Unlike the circular profile for the constant mean curvature equation,
a grim-reaper profile with slowly varying center and width cannot by
itself produce a nonconvex sublevel set while its zero channel remains
convex.  This elementary obstruction is recorded in
Proposition~\ref{prop:no-go}.

The obstruction is removed at the near-critical half-width
\(\pi/2-q\).  We introduce an explicit corrector which cancels the
complete term linear in the squared modulation speed.  With modulation
scale \(q^{9/4}\), the corrected model satisfies the weighted residual
estimate
\[
 |\Fop[V_q](x,y)|\leq Cq^3\sec^2y
\]
and has a finite midpoint gap of order \(q^{3/2}\).  The zero channel is completed
by smooth strictly convex caps.  Explicit upper and lower barriers then
approximate the zero-Dirichlet solution with error
\[
 O(q^2)+O\!\left(|\log q|e^{-cq^{-1/4}}\right)=o(q^{3/2})
\]
at the three points which detect the midpoint gap.

  This paper is organized as follows.
 In section 2, we   use a corrected grim-reaper profile on a long channel and it leads to a one-dimensional profile. In section 3 we  construct the strictly convex domain and solve the Dirichlet solution. In section 4, we use the comparison principle to detect the nonconvexity of the level set for the solution in section 3. In section 5, we complete the proof of the Theorem~\ref{thm:main}.

\section{The grim-reaper reduction and the corrector}
In this section,  we   use a corrected grim-reaper profile on a long channel and it leads to a one-dimensional profile, then we make the corrector.

Repeated indices are summed.  The constants \(c,C>0\) are independent
of the small parameter \(q\) and may change from line to line.

\subsection{The operator and the one-dimensional equation}

For \(p\in\R^2\), set
\[
 a^{ij}(p)=\delta_{ij}-\frac{p_ip_j}{1+|p|^2}.
\]
Equation \eqref{eq:intro-problem} is equivalent to
\[
 \Qop[u]:=a^{ij}(Du)u_{ij}=1.
\]
Writing the coordinates as \((x,y)\), define
\begin{align}
 \Fop[u]
 :={}&(1+u_y^2)u_{xx}-2u_xu_yu_{xy}
 +(1+u_x^2)u_{yy}\notag\\
 &\quad -(1+u_x^2+u_y^2).
 \label{eq:Fdef}
\end{align}
Then
\begin{equation}\label{eq:FQ}
 \Fop[u]=(1+|Du|^2)(\Qop[u]-1).
\end{equation}
Thus \(\Fop[u]\) and \(\Qop[u]-1\) have the same sign.

If \(u\) depends only on \(y\), the equation reduces to
\begin{equation}\label{eq:grim-ode}
 \frac{u_{yy}}{1+u_y^2}=1,
 \qquad\text{or equivalently}\qquad
 u_{yy}=1+u_y^2.
\end{equation}
Every solution of \eqref{eq:grim-ode}, up to translations, is
\[
 u(y)=-\log\cos(y-c)+C.
\]
Consequently, if a one-dimensional zero profile is placed on the
interval \(c-A<y<c+A\), then it is necessarily
\begin{equation}\label{eq:zeroth-profile}
 U_0(x,y)
 =-\log\cos\bigl(y-c(x)\bigr)+\log\cos A(x).
\end{equation}
This is the grim-reaper reduction.

\subsection{The zeroth-order obstruction}

The direct stretched-graph construction used for the constant mean
curvature equation suggests allowing the center and the half-width in
\eqref{eq:zeroth-profile} to vary slowly.  The following proposition
shows why this construction cannot work without a corrector.

\begin{proposition}\label{prop:no-go}
Let \(I\subset\R\) be an interval, let \(c,A\in C^2(I)\), and suppose
\(0<A(x)<\pi/2\).  If the channel
\[
 \mathcal D
 =\{(x,y):x\in I,\ c(x)-A(x)<y<c(x)+A(x)\}
\]
is convex, then every sublevel set of the function \(U_0\) in
\eqref{eq:zeroth-profile} is convex.
\end{proposition}

\begin{proof}
Convexity of the upper and lower boundary graphs gives
\[
 \frac{\partial^2 c}{\partial x^2}+\frac{\partial^2 A}{\partial x^2} \leq0,\qquad \frac{\partial^2 c}{\partial x^2}-\frac{\partial^2 A}{\partial x^2}\geq0.
\]
Hence
\begin{equation}\label{eq:center-width-curvature}
 \frac{\partial^2 A}{\partial x^2} \leq \frac{\partial^2 c}{\partial x^2} \leq-\frac{\partial^2 A}{\partial x^2},\qquad \frac{\partial^2 A}{\partial x^2}\leq0.
\end{equation}
Since \(U_0<0\) throughout \(\mathcal D\), one has
\(\{U_0<\ell\}=\mathcal D\) for every \(\ell\geq0\).  It remains to
consider \(\ell=-h<0\).
For \(h>0\), let
\[
 I_h^\sharp=\{x\in I:A(x)>\arccos(e^{-h})\}.
\]
This is an interval, since it is a strict superlevel set of the
concave function \(A\).  If \(I_h^\sharp\) is empty, the corresponding
sublevel set is empty.  Otherwise, for \(x\in I_h^\sharp\), the level
\(U_0=-h\) is bounded by
\[
 y=c(x)\pm R_h(A(x)),\qquad
 R_h(A)=\arccos(e^h\cos A).
\]
Direct differentiation gives
\[
 \frac{\partial R_h}{\partial A} =\frac{e^h\sin A}{\sin R_h}>1
\]
and
\[
  \frac{\partial^2 R_h}{\partial A^2}
 =\frac{e^h\cos A}{\sin R_h}
 \left(1-\frac{e^{2h}\sin^2A}{\sin^2R_h}\right)<0.
\]
Using \eqref{eq:center-width-curvature}, we obtain
\[
 \frac{\partial^2 (c+R_h(A))}{\partial x^2}
 \leq(\frac{\partial R_h}{\partial A}-1) \frac{\partial^2 A}{\partial x^2}+\frac{\partial^2 R_h}{\partial A^2} (\frac{\partial A}{\partial x})^2\leq0
\]
and
\[
 \frac{\partial^2 (c-R_h(A))}{\partial x^2}
 \geq(1-\frac{\partial R_h}{\partial A}) \frac{\partial^2 A}{\partial x^2}-\frac{\partial^2 R_h}{\partial A^2}(\frac{\partial A}{\partial x})^2\geq0.
\]
Thus the upper level graph is concave and the lower level graph is
convex on \(I_h^\sharp\).  More precisely,
\[
 \{U_0<-h\}
 =\{(x,y):x\in I_h^\sharp,\
 c(x)-R_h(A(x))<y<c(x)+R_h(A(x))\}.
\]
The right-hand side is convex: the projection interval
\(I_h^\sharp\) is convex, its upper boundary is concave, and its lower
boundary is convex.  This proves the assertion also in the equality
case \(A(x)=\arccos(e^{-h})\), where the strict sublevel set has no
point over \(x\).
\end{proof}

\subsection{The near-critical scale and the corrector}

Fix a small number \(0<\Theta<10^{-3}\).  For \(q>0\), define
\begin{equation}\label{eq:parameters}
 \gamma_0=\arccos q,\qquad
 \lambda=q^{9/4},\qquad
 \tau=\lambda^2=q^{9/2},\qquad
 L_q=\frac{\Theta q}{\lambda}=\Theta q^{-5/4}.
\end{equation}
Let
\[
 J_q=L_q+2,\qquad \Lambda_q=J_q+4,
\]
and, on \(|x|\leq\Lambda_q\), put
\begin{equation}\label{eq:gamma}
 \gamma_q(x)=\gamma_0+\lambda x-q^7x^2.
\end{equation}

Set
\[
 g(s)=-\log\cos s,\qquad
 \Acal(s)=s\tan s+\log\cos s.
\]
Then
\[
 g'(s)=\tan s,\qquad
 g''(s)=\sec^2s,\qquad
 \Acal'(s)=s\sec^2s.
\]

The following calculation motivates the choice of the corrector \(w\).
We first ignore the non-affine part of the width function and suppose that
\(\gamma\) is affine, with
\[
 (\gamma_x)^2=\tau,\qquad \gamma_{xx}=0.
\]
Consider the uncorrected profile
\[
 U_0(x,y)=g(y)-g(\gamma(x)).
\]
Writing
\[
 t=\tan y,\qquad S=\sec^2 y,\qquad
 T=\tan\gamma,\qquad \Sigma=\sec^2\gamma,
\]
a direct substitution into \eqref{eq:Fdef} gives
\begin{equation}\label{eq:uncorrected-residual}
 \Fop[U_0]=-\tau(S+T^2).
\end{equation}
Thus \(U_0\) fails to solve the equation by an error of order \(\tau\).

We now look for a correction of the form
\[
 U_0(x,y)+\tau w(\gamma(x),y).
\]
Expanding the residual to first order in \(\tau\), we find that the
coefficient of \(\tau\) is
\[
 w_{yy}-2\tan y\,w_y-\bigl(\sec^2y+\tan^2\gamma\bigr).
\]
Consequently, the entire first-order error is canceled if \(w\) satisfies
\begin{equation}\label{eq:corrector-ode}
 w_{yy}-2\tan y\,w_y
 =\sec^2y+\tan^2\gamma,
 \qquad
 w(\gamma,\pm\gamma)=0.
\end{equation}
The boundary conditions are imposed so that the correction does not
alter the zero set \(y=\pm\gamma\). An explicit solution is
\begin{equation}\label{eq:wdef}
 w(\gamma,y)
 =\Acal(y)-\Acal(\gamma)
 +\frac{\tan^2\gamma}{2}
 \bigl(y\tan y-\gamma\tan\gamma\bigr).
\end{equation}

Returning to the actual width function \(\gamma_q\), we define the
corrected profile by
\begin{equation}\label{eq:Vdef}
 V_q(x,y)
 =g(y)-g(\gamma_q(x))
 +\tau w(\gamma_q(x),y).
\end{equation}
Since both \(\Acal(s)\) and \(s\tan s\) are even and strictly increasing
on \([0,\pi/2)\), the corrected profile has the same zero boundary and
the same sign in the channel:
\begin{equation}\label{eq:Vboundary}
 V_q(x,\pm\gamma_q(x))=0,
 \qquad
 V_q(x,y)<0
 \quad\text{whenever } |y|<\gamma_q(x).
\end{equation}

\section{The convex domain and the Dirichlet solution}
\label{sec:domain}
  In this section we construct the strictly convex domain and solve the Dirichlet solution.
\subsection{Closing the lateral zero arcs}

Let
\[
 \mathcal C_q
 =\{(x,y):|x|\leq\Lambda_q,\ |y|\leq\gamma_q(x)\}
\]
be the extended channel.  By \eqref{eq:Vboundary}, the zero set of
\(V_q\) in \(\mathcal C_q\) consists of the two lateral arcs
\[
 y=\pm\gamma_q(x),\qquad |x|\leq\Lambda_q.
\]
We retain the central parts of these arcs and complete them by strictly
convex caps.  The caps lie in \(\{V_q<0\}\); thus the completed curve is
not asserted to be the zero set of \(V_q\).  It will be the zero
boundary for the Dirichlet problem constructed below by Proposition \ref{prop:existence}.

At \(x=J_q\), let
\[
 Y_+=\gamma_q(J_q),\qquad m_+=\gamma_q'(J_q),
\]
and define
\[
 c_+=J_q+m_+Y_+,\qquad
 R_+=Y_+\sqrt{1+m_+^2}.
\]
The upper graph of the right circle
\[
 (x-c_+)^2+y^2=R_+^2
\]
is
\[
 \beta_+(x)=\sqrt{R_+^2-(x-c_+)^2}.
\]
It satisfies
\[
 \beta_+(J_q)=Y_+,\qquad
 \beta_+'(J_q)=m_+,\qquad
 \beta_+''=-\frac{R_+^2}{\beta_+^3}\leq-c.
\]
Since \(\gamma_q''=-2q^7\), tangency at \(J_q\) gives
\begin{equation}\label{eq:cap-below-channel}
 \beta_+(x)\leq\gamma_q(x)
\end{equation}
on the graphical part of the cap.  The right endpoint has
\(x\)-coordinate \(c_++R_+\).  Directly from the definitions,
\[
 Y_+=\frac{\pi}{2}+O(q),\qquad m_+=O(q^{9/4}),
\]
and therefore
\[
 (c_++R_+)-J_q
 =m_+Y_++Y_+\sqrt{1+m_+^2}
 =\frac{\pi}{2}+O(q)<3
\]
for all sufficiently small \(q\).  Thus the right endpoint is less
than \(J_q+3<\Lambda_q\).  The left cap is constructed by applying the same
argument to \(\widehat\gamma_q(X)=\gamma_q(-X)\) and reflecting back.

We use the following interpolation lemma at the two junctions.

\begin{lemma}[Two-moment smoothing]\label{lem:smoothing}
Let \(a<J<b\).  Let \(\gamma\in C^\infty([a,b])\) and
\(\beta\in C^\infty([J,b])\) be concave functions satisfying
\[
 \gamma(J)=\beta(J),\qquad
 \gamma'(J)=\beta'(J).
\]
Suppose that
\[
 \gamma''\equiv-\kappa<0\quad\text{on }[a,b],
 \qquad
 \beta''\leq\gamma''\quad\text{on }[J,b],
\]
and that the latter inequality is strict on a nonempty subinterval of
\((J,b)\). Then there exists \(f\in C^\infty([a,b])\) which agrees
with \(\gamma\) near \(a\), agrees with \(\beta\) near \(b\), and
satisfies
\[
 f''<0,\qquad f\leq\gamma\quad\text{on }[a,b].
\]
\end{lemma}

\begin{proof}
Consider the piecewise second derivative
\[
 m_0(t)=
 \begin{cases}
  \gamma''(t),&a\leq t<J,\\
  \beta''(t),&J<t\leq b.
 \end{cases}
\]
The matching conditions at \(J\) imply
\begin{align}
 \int_a^b m_0(t)\,\dd t
 &=\beta'(b)-\gamma'(a),
 \label{eq:smoothing-moment-zero}\\
 \int_a^b(b-t)m_0(t)\,\dd t
 &=\beta(b)-\gamma(a)-(b-a)\gamma'(a).
 \label{eq:smoothing-moment-one}
\end{align}
Indeed, the first identity follows by integrating separately over
\([a,J]\) and \([J,b]\). For the second one, integration by parts on
these two intervals gives
\begin{align*}
 \int_a^b(b-t)m_0(t)\,\dd t
 ={}&\beta(b)-\gamma(a)-(b-a)\gamma'(a)\\
 &+\gamma(J)-\beta(J)
 +(b-J)\bigl(\gamma'(J)-\beta'(J)\bigr),
\end{align*}
and the last two terms vanish.

We record explicitly why the two identities
\eqref{eq:smoothing-moment-zero}--\eqref{eq:smoothing-moment-one}
must be preserved in the smoothing procedure. Given a smooth function
\(m\), reconstruct a profile from the left endpoint by setting
\begin{equation}\label{eq:smoothing-reconstruction}
 f_m(x)=\gamma(a)+(x-a)\gamma'(a)
 +\int_a^x(x-t)m(t)\,\dd t.
\end{equation}
Then
\[
 f_m'(x)=\gamma'(a)+\int_a^x m(t)\,\dd t,
 \qquad
 f_m''(x)=m(x).
\]
Consequently,
\begin{align}
 f_m'(b)-\beta'(b)
 &=\int_a^b\bigl(m-m_0\bigr)(t)\,\dd t,
 \label{eq:smoothing-slope-error}\\
 f_m(b)-\beta(b)
 &=\int_a^b(b-t)\bigl(m-m_0\bigr)(t)\,\dd t.
 \label{eq:smoothing-value-error}
\end{align}
If, in addition, \(m=\beta''\) near \(b\), then
\(f_m-\beta\) is affine there. It therefore vanishes near \(b\) if
and only if the right-hand sides of
\eqref{eq:smoothing-slope-error} and
\eqref{eq:smoothing-value-error} both vanish. Thus the two moment
conditions are precisely what is needed to recover both the value and
the slope of \(\beta\) at the right endpoint.

We now construct a smooth second derivative with these two moments.
Choose \(\vartheta\in C^\infty(\mathbb R)\) such that
\[
 0\leq\vartheta\leq1,\qquad
 \vartheta=0\quad\text{on }(-\infty,0],\qquad
 \vartheta=1\quad\text{on }[1,\infty),
\]
and set
\[
 \widetilde m_\varepsilon(t)=
 \begin{cases}
  \gamma''(t),&t\leq J,\\[1mm]
  \gamma''(t)+
  \vartheta\!\left(\dfrac{t-J}{\varepsilon}\right)
  \bigl(\beta''(t)-\gamma''(t)\bigr),&t\geq J.
 \end{cases}
\]
Then \(\widetilde m_\varepsilon\in C^\infty([a,b])\),
\[
 \widetilde m_\varepsilon\leq\gamma'',
\]
and \(\widetilde m_\varepsilon\) agrees with \(\gamma''\) near \(a\)
and with \(\beta''\) on \([J+\varepsilon,b]\). Define its two moment
errors by
\begin{align*}
 E_{0,\varepsilon}
 &=\int_a^b
   \bigl(\widetilde m_\varepsilon-m_0\bigr)(t)\,\dd t,\\
 E_{1,\varepsilon}
 &=\int_a^b(b-t)
   \bigl(\widetilde m_\varepsilon-m_0\bigr)(t)\,\dd t.
\end{align*}
Since \(\widetilde m_\varepsilon-m_0\) is uniformly bounded and
supported in \([J,J+\varepsilon]\), one has
\begin{equation}\label{eq:smoothing-errors}
 |E_{0,\varepsilon}|+|E_{1,\varepsilon}|
 \leq C\varepsilon.
\end{equation}

By the strict inequality assumption, there exist two ordered compact
intervals $ I_1, I_2 \Subset(J,b)$  such that \(I_1\) lies strictly to the left of \(I_2\)
and a constant \(\delta>0\) such that
\begin{equation}\label{eq:smoothing-gap}
 \beta''\leq\gamma''-3\delta
 \quad\text{on }I_1\cup I_2.
\end{equation}
Choose nonnegative functions
\[
 \varphi_i\in C_c^\infty(I_i),\qquad
 \int_a^b\varphi_i(t)\,\dd t=1,
\]
and write
\[
 \mu_i=\int_a^b t\varphi_i(t)\,\dd t.
\]
Since  \(I_1\) lies strictly to the left of \(I_2\), we have \(\mu_1<\mu_2\). We seek
\[
 m_\varepsilon
 =\widetilde m_\varepsilon
 +c_1(\varepsilon)\varphi_1
 +c_2(\varepsilon)\varphi_2.
\]
The two required moment identities 
\begin{align*}
    \int_a^b m_{\epsilon}(t) dt &=\int_a^b m_{0}(t) dt\\
    \int_a^b (b-t)m_{\epsilon}(t) dt &=\int_a^b (b-t)m_{0}(t) dt
\end{align*}
 are equivalent to
\begin{equation}\label{eq:smoothing-linear-system}
 \begin{pmatrix}
  1&1\\
  b-\mu_1&b-\mu_2
 \end{pmatrix}
 \begin{pmatrix}
  c_1(\varepsilon)\\
  c_2(\varepsilon)
 \end{pmatrix}
 =-
 \begin{pmatrix}
  E_{0,\varepsilon}\\
  E_{1,\varepsilon}
 \end{pmatrix}.
\end{equation}
The determinant of this matrix is
\(\mu_1-\mu_2\neq0\). More explicitly, the unique choice is
\begin{align}
 c_1(\varepsilon)
 &=
 \frac{(b-\mu_2)E_{0,\varepsilon}-E_{1,\varepsilon}}
 {\mu_2-\mu_1},
 \label{eq:smoothing-c1}\\
 c_2(\varepsilon)
 &=
 \frac{E_{1,\varepsilon}-(b-\mu_1)E_{0,\varepsilon}}
 {\mu_2-\mu_1}.
 \label{eq:smoothing-c2}
\end{align}
Hence \(m_\varepsilon\) has exactly the two moments of \(m_0\).
Moreover, \eqref{eq:smoothing-errors} and
\eqref{eq:smoothing-c1}--\eqref{eq:smoothing-c2} yield
\[
 |c_1(\varepsilon)|+|c_2(\varepsilon)|
 \leq C\varepsilon.
\]

We finally choose \(\varepsilon>0\) sufficiently small that
\[
 J+\varepsilon<\inf I_1,\qquad
 |c_i(\varepsilon)|\|\varphi_i\|_{L^\infty}\leq\delta
 \quad (i=1,2).
\]
On \(I_i\), the supports are disjoint and
\(\widetilde m_\varepsilon=\beta''\); hence
\eqref{eq:smoothing-gap} gives
\[
 m_\varepsilon
 \leq\beta''+\delta
 \leq\gamma''-2\delta.
\]
Outside the two bump supports,
\[
 m_\varepsilon=\widetilde m_\varepsilon\leq\gamma''.
\]
It follows that
\begin{equation}\label{eq:smoothing-concavity}
 m_\varepsilon\leq\gamma''=-\kappa<0
 \quad\text{on }[a,b].
\end{equation}

Fix such an \(\varepsilon\), write \(m=m_\varepsilon\), and define
\(f=f_m\) by \eqref{eq:smoothing-reconstruction}. Then \(f''=m<0\).
Since \(m=\gamma''\) near \(a\), the initial conditions in
\eqref{eq:smoothing-reconstruction} imply that \(f=\gamma\) there.
The restored moments and
\eqref{eq:smoothing-slope-error}--\eqref{eq:smoothing-value-error}
give
\[
 f'(b)=\beta'(b),\qquad f(b)=\beta(b).
\]
Since \(m=\beta''\) near \(b\), it follows that \(f=\beta\) near \(b\).
Finally,
\[
 f(x)-\gamma(x)
 =\int_a^x(x-t)\bigl(m(t)-\gamma''(t)\bigr)\,\dd t
 \leq0
\]
by \eqref{eq:smoothing-concavity}. Thus \(f\leq\gamma\) on
\([a,b]\), as required.
\end{proof}

We next verify that the smoothing may be performed on neighborhoods
whose size is independent of \(q\).  We give the argument at the right
junction; the left one is identical after reflection.  In the translated
coordinate \(z=x-J_q\), one has, uniformly for \(z\) in a fixed bounded
interval,
\[
 \gamma_q(J_q+z)=\frac{\pi}{2}+O(q),\qquad
 \gamma_q'(J_q+z)=O(q^{9/4}),\qquad
 \gamma_q''=-2q^7.
\]
Consequently,
\[
 Y_+\longrightarrow\frac{\pi}{2},\qquad
 m_+\longrightarrow0,\qquad
 R_+\longrightarrow\frac{\pi}{2}.
\]
Choose a fixed \(0<\sigma<1\) sufficiently small.  For all sufficiently
small \(q\), the circular graph \(\beta_+\) is defined on
\([J_q,J_q+\sigma]\), and its explicit second derivative gives
\[
 -C_*\leq\beta_+''\leq-c_*<0
 \quad\text{on }[J_q,J_q+\sigma],
\]
where \(c_*,C_*\) are independent of \(q\).  Hence, after decreasing
\(q\) if necessary,
\[
 \beta_+''\leq\gamma_q''-\frac{c_*}{2}
 \quad\text{on }[J_q,J_q+\sigma].
\]
Apply Lemma~\ref{lem:smoothing} on
\([J_q-\sigma,J_q+\sigma]\).  In its proof, choose the two bump
intervals by translating two fixed disjoint subintervals of
\((0,\sigma)\).  Their separation, the inverse of the moment matrix,
the \(C^0\)-bounds for the second derivatives, and the strict gap are
then all uniform in \(q\).  Thus the smoothing parameter
\(\varepsilon\) can be chosen independently of \(q\).

It remains to check that the smoothed half-width stays positive.
Put \(a=J_q-\sigma\) and \(b=J_q+\sigma\).  The interpolation agrees
with \(\gamma_q\) near \(a\) and with \(\beta_+\) near \(b\).  For the
fixed sufficiently small \(\sigma\) chosen above, both
\(\gamma_q(a)\) and \(\beta_+(b)\) have a positive lower bound
\(c_\sigma\) independent of small \(q\).  Since \(f_m''<0\), concavity
gives
\[
 f_m((1-t)a+tb)
 \geq(1-t)f_m(a)+tf_m(b)
 \geq c_\sigma
 \qquad
 (0\leq t\leq1).
\]
Thus the interpolation remains strictly above the symmetry axis.
The unchanged circular cap is positive before its terminal point.
Applying the same construction at the left junction and reflecting
across \(y=0\), we obtain an
interval \((\alpha_q,\beta_q)\) and a positive concave half-width
\(f_q\) such that
\begin{equation}\label{eq:domain-graph}
 \Om_q
 =\{(x,y):\alpha_q<x<\beta_q,\ |y|<f_q(x)\}
\end{equation}
is bounded and smooth, and
\begin{align}
 \Om_q\cap\{|x|\leq L_q\}
 &=\{(x,y):|x|\leq L_q,\ |y|<\gamma_q(x)\},
 \label{eq:central-domain}\\
 \Om_q&\subset
 \{(x,y):|x|\leq\Lambda_q,\ |y|<\gamma_q(x)\}.
 \label{eq:domain-inclusion}
\end{align}
At each terminal point the upper and lower arcs belong to the same
circle and therefore join smoothly.

For completeness, the concavity of \(f_q\) gives the convexity of
\(\Om_q\) directly.  If \(P_i=(x_i,y_i)\in\Om_q\), \(i=0,1\), and
\(0<s<1\), then
\begin{align*}
 |(1-s)y_0+sy_1|
 &\leq(1-s)|y_0|+s|y_1|\\
 &<(1-s)f_q(x_0)+sf_q(x_1)\\
 &\leq f_q\bigl((1-s)x_0+sx_1\bigr).
\end{align*}
Hence the segment joining \(P_0\) and \(P_1\) is contained in
\(\Om_q\).  The unchanged channel arcs, circular caps, and interpolating
arcs all have strictly positive curvature.  Therefore, for each fixed
sufficiently small \(q>0\), smoothness and compactness of
\(\partial\Om_q\) imply
that its curvature has a positive minimum \(\kappa_q>0\).  In the
fixed-\(q\) sense, \(\Om_q\) is uniformly convex.  No lower bound for
\(\kappa_q\) uniform in \(q\) is asserted or needed.

We finally record the localization needed in the boundary comparison.
On the right, the interpolation is contained in
\([J_q-\sigma,J_q+\sigma]\), while the circular cap terminates before
\(J_q+3\).  Since \(\Lambda_q=J_q+4\), every point on a right cap or
interpolation arc satisfies
\[
 \Lambda_q-(4+\sigma)\leq x\leq\Lambda_q.
\]
The reflected construction gives the same estimate on the left.
Consequently, all cap and interpolation parts of
\(\partial\Om_q\) satisfy
\begin{equation}\label{eq:cap-localization}
 \Lambda_q-C_0\leq |x|\leq\Lambda_q
\end{equation}
with \(C_0=4+\sigma\), independent of \(q\).

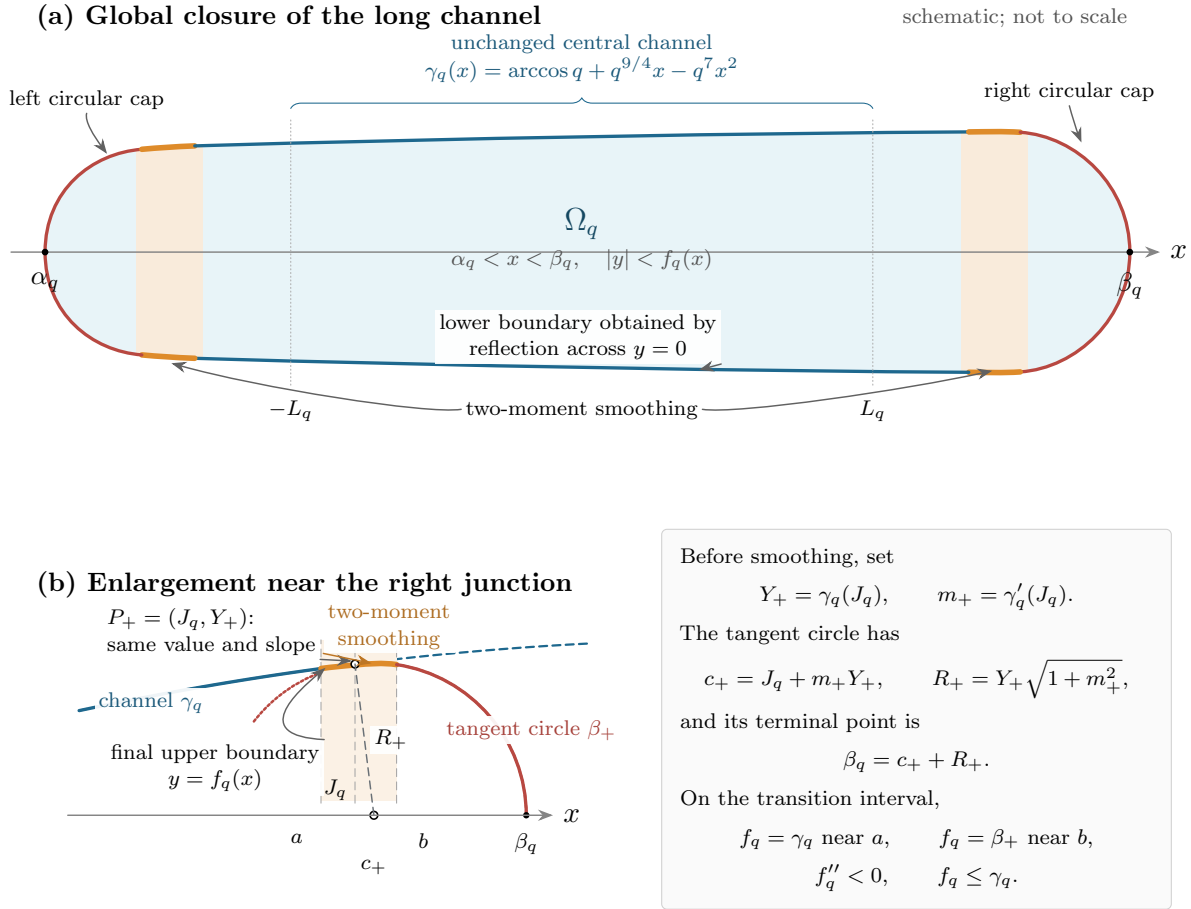
\begin{figure}[ht]
\centering
\begin{tikzpicture}[
  x=1cm,
  y=1cm,
  line cap=round,
  line join=round,
  >=Stealth,
  channel/.style={draw=ChannelBlue,line width=1.25pt},
  capline/.style={draw=CapRed,line width=1.25pt},
  smoothing/.style={draw=SmoothOrange,line width=2.1pt},
  guide/.style={draw=GuideGray,densely dashed,line width=.55pt},
  axis/.style={draw=black!48,line width=.55pt,-{Stealth[length=2mm]}},
  callout/.style={
    -{Stealth[length=2mm]},
    draw=black!62,
    line width=.55pt
  },
  labelbox/.style={
    fill=white,
    fill opacity=.93,
    text opacity=1,
    inner sep=1.5pt,
    font=\scriptsize
  }
]

%=====================================================================
% (a) Global construction.
% The proportions are deliberately schematic: in the proof,
% J_q is much larger than the O(1) height and cap lengths.
%=====================================================================
\begin{scope}[shift={(0,6.55)}]

  % The completed domain.
  \path[fill=DomainFill,draw=none]
    (-7.10,0)
    .. controls (-7.10,.82) and (-6.48,1.31) .. (-5.82,1.36)
    .. controls (-5.60,1.38) and (-5.38,1.395) .. (-5.12,1.405)
    .. controls (-1.75,1.505) and (2.35,1.575) .. (5.12,1.590)
    .. controls (5.35,1.592) and (5.57,1.600) .. (5.80,1.585)
    .. controls (6.55,1.535) and (7.25,.87) .. (7.25,0)
    .. controls (7.25,-.87) and (6.55,-1.535) .. (5.80,-1.585)
    .. controls (5.57,-1.600) and (5.35,-1.592) .. (5.12,-1.590)
    .. controls (2.35,-1.575) and (-1.75,-1.505) .. (-5.12,-1.405)
    .. controls (-5.38,-1.395) and (-5.60,-1.38) .. (-5.82,-1.36)
    .. controls (-6.48,-1.31) and (-7.10,-.82) .. (-7.10,0)
    -- cycle;

  % Highlight the two short regions modified by smoothing.
  \begin{scope}
    \clip
      (-7.10,0)
      .. controls (-7.10,.82) and (-6.48,1.31) .. (-5.82,1.36)
      .. controls (-5.60,1.38) and (-5.38,1.395) .. (-5.12,1.405)
      .. controls (-1.75,1.505) and (2.35,1.575) .. (5.12,1.590)
      .. controls (5.35,1.592) and (5.57,1.600) .. (5.80,1.585)
      .. controls (6.55,1.535) and (7.25,.87) .. (7.25,0)
      .. controls (7.25,-.87) and (6.55,-1.535) .. (5.80,-1.585)
      .. controls (5.57,-1.600) and (5.35,-1.592) .. (5.12,-1.590)
      .. controls (2.35,-1.575) and (-1.75,-1.505) .. (-5.12,-1.405)
      .. controls (-5.38,-1.395) and (-5.60,-1.38) .. (-5.82,-1.36)
      .. controls (-6.48,-1.31) and (-7.10,-.82) .. (-7.10,0)
      -- cycle;
    \fill[SmoothOrange!17] (-5.90,-1.75) rectangle (-5.02,1.75);
    \fill[SmoothOrange!17] (5.02,-1.75) rectangle (5.90,1.75);
  \end{scope}

  % Upper boundary.
  \draw[capline]
    (-7.10,0)
    .. controls (-7.10,.82) and (-6.48,1.31) .. (-5.82,1.36);
  \draw[smoothing]
    (-5.82,1.36)
    .. controls (-5.60,1.38) and (-5.38,1.395) .. (-5.12,1.405);
  \draw[channel]
    (-5.12,1.405)
    .. controls (-1.75,1.505) and (2.35,1.575) .. (5.12,1.590);
  \draw[smoothing]
    (5.12,1.590)
    .. controls (5.35,1.592) and (5.57,1.600) .. (5.80,1.585);
  \draw[capline]
    (5.80,1.585)
    .. controls (6.55,1.535) and (7.25,.87) .. (7.25,0);

  % Lower boundary: reflection across y=0.
  \draw[capline]
    (7.25,0)
    .. controls (7.25,-.87) and (6.55,-1.535) .. (5.80,-1.585);
  \draw[smoothing]
    (5.80,-1.585)
    .. controls (5.57,-1.600) and (5.35,-1.592) .. (5.12,-1.590);
  \draw[channel]
    (5.12,-1.590)
    .. controls (2.35,-1.575) and (-1.75,-1.505) .. (-5.12,-1.405);
  \draw[smoothing]
    (-5.12,-1.405)
    .. controls (-5.38,-1.395) and (-5.60,-1.38) .. (-5.82,-1.36);
  \draw[capline]
    (-5.82,-1.36)
    .. controls (-6.48,-1.31) and (-7.10,-.82) .. (-7.10,0);

  % Axis, endpoints, and the preserved central interval.
  \draw[axis] (-7.55,0)--(7.65,0) node[right] {\(x\)};
  \fill (-7.10,0) circle[radius=1.25pt];
  \fill (7.25,0) circle[radius=1.25pt];
  \node[below=4pt,font=\small] at (-7.10,0) {\(\alpha_q\)};
  \node[below=4pt,font=\small] at (7.25,0) {\(\beta_q\)};

  \draw[black!32,densely dotted] (-3.85,-1.76)--(-3.85,1.76);
  \draw[black!32,densely dotted] (3.85,-1.76)--(3.85,1.76);
  \node[below=2pt,font=\scriptsize] at (-3.85,-1.76) {\(-L_q\)};
  \node[below=2pt,font=\scriptsize] at (3.85,-1.76) {\(L_q\)};

  % Labels and construction arrows.
  \node[font=\bfseries\small,anchor=west] at (-7.35,3.10)
    {(a) Global closure of the long channel};
  \node[font=\scriptsize,text=black!68,anchor=east] at (7.35,3.10)
    {schematic; not to scale};

  \draw[
    decorate,
    decoration={brace,amplitude=4pt,raise=2pt},
    draw=ChannelBlue
  ]
    (-3.85,1.83)--(3.85,1.83)
    node[midway,above=8pt,align=center,font=\scriptsize,text=ChannelBlue!85!black]
    {unchanged central channel\\
     \(\displaystyle
       \gamma_q(x)=\arccos q+q^{9/4}x-q^7x^2\)};

  \node[font=\large,text=ChannelBlue!68!black] at (0,.38)
    {\(\Omega_q\)};
  \node[font=\scriptsize,text=black!68] at (0,-.10)
    {\(\alpha_q<x<\beta_q,\quad |y|<f_q(x)\)};

  \node[labelbox,align=center] (leftcaptext) at (-6.55,2.02)
    {left circular cap};
  \draw[callout] (leftcaptext.south)
    .. controls (-6.58,1.73) and (-6.45,1.53) .. (-6.28,1.30);

  \node[labelbox,align=center] (rightcaptext) at (6.45,2.12)
    {right circular cap};
  \draw[callout] (rightcaptext.south)
    .. controls (6.48,1.82) and (6.55,1.60) .. (6.61,1.45);

  \node[labelbox,align=center] (smoothtext) at (0,-2.08)
    {two-moment smoothing};
  \draw[callout] (smoothtext.west)
    .. controls (-2.05,-2.15) and (-4.30,-1.79) .. (-5.43,-1.47);
  \draw[callout] (smoothtext.east)
    .. controls (2.05,-2.15) and (4.30,-1.82) .. (5.43,-1.57);

  \node[labelbox,align=center] (reflectiontext) at (-.05,-1.12)
    {lower boundary obtained by\\reflection across \(y=0\)};
  \draw[callout] (reflectiontext.south east)--(1.55,-1.56);
\end{scope}

%=====================================================================
% (b) Enlargement of the right junction.
% The dashed and dotted curves form the original C^1 tangent splice.
% The orange curve represents the final two-moment interpolation.
%=====================================================================
\begin{scope}[shift={(0,-.90)}]

  \node[font=\bfseries\small,anchor=west] at (-7.35,3.06)
    {(b) Enlargement near the right junction};

  % Reference channel.  In local coordinates P_+=(J_q,Y_+) is
  % represented by (-3,2), and the visible slope is exaggerated.
  \draw[draw=ChannelBlue,densely dashed,line width=.85pt]
    plot[domain=-6.65:.10,samples=90,smooth]
    (\x,{2+.125*(\x+3)-.012*(\x+3)*(\x+3)});

  % Reference tangent circle with center C_+=(-2.75,0) and
  % radius sqrt(4.0625).  It is drawn as a true circle.
  \pgfmathsetmacro{\Rcap}{sqrt(4.0625)}
  \draw[draw=CapRed,densely dotted,line width=.95pt]
    ({-2.75+\Rcap*cos(142)},{\Rcap*sin(142)})
    arc[start angle=142,end angle=0,radius=\Rcap];

  % Highlight the transition interval [a,b].
  \fill[SmoothOrange!12] (-3.45,.20) rectangle (-2.45,2.32);
  \draw[black!35,densely dashed] (-3.45,.16)--(-3.45,2.32);
  \draw[black!35,densely dashed] (-3.00,.16)--(-3.00,2.32);
  \draw[black!35,densely dashed] (-2.45,.16)--(-2.45,2.32);

  % Final profile f_q: channel near a, smooth interpolation on [a,b],
  % and circular cap near b.  The interpolation is schematic.
  \draw[channel]
    plot[domain=-6.65:-3.45,samples=45,smooth]
    (\x,{2+.125*(\x+3)-.012*(\x+3)*(\x+3)});

  \draw[smoothing]
    (-3.45,1.9413)
    .. controls (-3.30,1.9617) and (-3.15,1.9752) .. (-3.00,1.9885)
    .. controls (-2.8167,2.0047) and (-2.6333,2.0207) .. (-2.45,1.9931);

  \pgfmathsetmacro{\angleb}{atan2(1.9931,.30)}
  \draw[capline]
    ({-2.75+\Rcap*cos(\angleb)},{\Rcap*sin(\angleb)})
    arc[start angle=\angleb,end angle=0,radius=\Rcap];

  % Tangency point of the unsmoothed profiles and the circle center.
  \filldraw[fill=white,draw=black,line width=.6pt]
    (-3,2) circle[radius=1.45pt];
  \filldraw[fill=white,draw=black,line width=.6pt]
    (-2.75,0) circle[radius=1.45pt];
  \fill (-.7344,0) circle[radius=1.2pt];
  \draw[guide] (-2.75,0)--(-3,2)
    node[midway,right=2pt,labelbox] {\(R_+\)};

  % Local axis and labels.
  \draw[axis] (-6.80,0)--(-.38,0) node[right] {\(x\)};
  \node[below left=3pt,font=\scriptsize] at (-3.45,0) {\(a\)};
  \node[above left=2pt,font=\scriptsize] at (-2.90,0) {\(J_q\)};
  \node[below right=3pt,font=\scriptsize] at (-2.40,0) {\(b\)};
  \node[below=13pt,font=\scriptsize] at (-2.75,0) {\(c_+\)};
  \node[below=3pt,font=\scriptsize] at (-.7344,0) {\(\beta_q\)};

  \node[labelbox,anchor=west,text=ChannelBlue!85!black]
    at (-6.45,1.49) {channel \(\gamma_q\)};
  \node[labelbox,anchor=west,text=CapRed!90!black]
    at (-1.85,1.13) {tangent circle \(\beta_+\)};

  \node[labelbox,align=left] (tangenttext) at (-4.90,2.43)
    {\(P_+=(J_q,Y_+)\):\\same value and slope};
  \draw[callout] (tangenttext.south east)--(-3.04,2.04);

  \node[labelbox,align=center] (finaltext) at (-4.85,.63)
    {final upper boundary\\\(y=f_q(x)\)};
  \draw[callout] (finaltext.north east)
    .. controls (-4.42,1.05) and (-3.76,1.64) .. (-3.39,1.95);

  \node[
    labelbox,
    text=SmoothOrange!78!black,
    align=center
  ] (localSmooth) at (-2.55,2.49)
    {two-moment\\smoothing};
  \draw[callout,draw=SmoothOrange!85!black]
    (localSmooth.south west)--(-2.78,2.01);

  % Exact formulas used in the construction.
  \node[
    anchor=west,
    align=left,
    text width=6.25cm,
    rounded corners=2pt,
    draw=black!18,
    fill=black!2,
    inner sep=7pt,
    font=\scriptsize
  ] at (1.05,1.27) {%
    Before smoothing, set
    \[
      Y_+=\gamma_q(J_q),\qquad
      m_+=\gamma_q'(J_q).
    \]
    The tangent circle has
    \[
      c_+=J_q+m_+Y_+,\qquad
      R_+=Y_+\sqrt{1+m_+^2},
    \]
    and its terminal point is
    \[
      \beta_q=c_++R_+.
    \]
    On the transition interval,
    \[
      f_q=\gamma_q\ \text{near }a,\qquad
      f_q=\beta_+\ \text{near }b,
    \]
    \[
      f_q''<0,\qquad f_q\leq\gamma_q.
    \]
  };

\end{scope}

\end{tikzpicture}

\caption{Schematic construction of \(\Omega_q\).}
\label{fig:omega-q-construction}
\end{figure}

\subsection{The zero-Dirichlet solution}

The preceding construction gives the domain before any comparison
argument is used.  We now solve the boundary value problem on that
domain.

\begin{proposition}\label{prop:existence}
For every sufficiently small \(q\), there is a unique function
\[
 u_q\in C^\infty(\Om_q)\cap C(\overline{\Om_q})
\]
such that
\begin{equation}\label{eq:uq-problem}
 \begin{cases}
 \Qop[u_q]=1&\text{in }\Om_q,\\
 u_q=0&\text{on }\partial\Om_q.
 \end{cases}
\end{equation}
\end{proposition}

\begin{proof}
The domain \(\Om_q\) is bounded and \(C^\infty\), and its boundary
curvature is strictly positive.  In particular, it is mean convex in
the outward-normal convention used in
\cite[Theorem~C.2]{Zhou2024}.  Taking \(n=2\) and \(\alpha=1\) in that
theorem, its translating mean curvature equation is precisely
\[
 \operatorname{div}\!\left(
 \frac{Du}{\sqrt{1+|Du|^2}}
 \right)
 =\frac{1}{\sqrt{1+|Du|^2}}.
\]
The cited existence and uniqueness result therefore gives
\(u_q\in C^2(\Om_q)\cap C(\overline{\Om_q})\).  Standard interior
regularity gives \(u_q\in C^\infty(\Om_q)\).
\end{proof}

\section{Comparison with the corrected profile}
\label{sec:comparison}
 In this  section, we use the comparison principle to detect the nonconvexity of the level set for the solution in section 3.
 
\subsection{The comparison principle}

\begin{lemma}[Comparison principle]\label{lem:comparison}
Let \(D\Subset\R^2\) be bounded, and let
\(u,v\in C^2(D)\cap C(\overline D)\).  If
\[
 \Qop[u]\geq\Qop[v]\quad\text{in }D,
 \qquad
 u\leq v\quad\text{on }\partial D,
\]
then \(u\leq v\) in \(D\).
\end{lemma}

\begin{proof}
Set \(z=u-v\) and \(u_t=v+tz\).  Linearization along this segment gives
\[
 \Qop[u]-\Qop[v]
 =A^{ij}z_{ij}+B^kz_k,
\]
where
\[
 A^{ij}=\int_0^1a^{ij}(Du_t)\,\dd t,
 \qquad
 B^k=\int_0^1
 \partial_{p_k}a^{ij}(Du_t)(u_t)_{ij}\,\dd t.
\]
Suppose that \(m=\max_{\overline D}z>0\).  Since \(z\leq0\) on
\(\partial D\), the set
\[
 D_m=\{z>m/2\}
\]
is compactly contained in \(D\).  On every component of \(D_m\), the
coefficients \(A^{ij}\) and \(B^k\) are bounded, and \(A^{ij}\) is
uniformly positive definite.  The strong maximum principle excludes
an interior maximum equal to \(m\), a contradiction.
\end{proof}

\subsection{Direct differential barriers}

The function used to move the corrected profile in the transverse
direction is
\begin{equation}\label{eq:Psi-def}
 \Psi_q(x,y)=\Acal(\gamma_q(x))-\Acal(y).
\end{equation}
It satisfies
\begin{equation}\label{eq:Psi-sign}
 \Psi_q>0\quad\text{in the interior of }\mathcal C_q,
 \qquad
 \Psi_q=0\quad\text{on }y=\pm\gamma_q(x).
\end{equation}

\begin{theorem}
\label{thm:direct-barriers}
There are constants \(K>1\) and \(\varepsilon_*>0\) with the following
property.  Fix \(0<\varepsilon_0\leq\varepsilon_*\) and \(C_E>0\), and
put
\begin{equation}\label{eq:barrier-parameters}
 a_q=Kq^3,\qquad
 e_q(x)=C_Eq|\log q|\,
 \frac{\cosh(\varepsilon_0qx)}
 {\cosh(\varepsilon_0q\Lambda_q)}.
\end{equation}
Then, for all sufficiently small \(q\), depending on
\(\Theta,K,\varepsilon_0,C_E\),
\begin{equation}\label{eq:direct-barriers}
 \Qop[V_q-a_q\Psi_q]\geq1,\qquad
 \Qop[V_q+(a_q+e_q)\Psi_q]\leq1
\end{equation}
throughout \(\mathcal C_q\).
\end{theorem}

\begin{proof}
We begin with the two signs in \eqref{eq:direct-barriers} and reduce
them directly to the estimates that will be verified below.  At a
point of \(\mathcal C_q\), write
\[
 \gamma=\gamma_q(x),\qquad
 t=\tan y,\quad S=\sec^2y,\qquad
 T=\tan\gamma,\quad\Sigma=\sec^2\gamma,
\]
and set
\[
 X=|y|S.
\]
For a function \(s=s(x)\), put
\[
 Z=t-syS
\]
and
\begin{equation}\label{eq:Bs-direct}
 B_s
 =2\Psi_qZyS+\Psi_q^2
 \bigl[S(1-s(1+2yt))-1\bigr].
\end{equation}

The exact computation carried out below gives
\begin{align}
 \Fop[V_q+s\Psi_q]
 ={}&\Fop[V_q]-sS-s^2X^2
 +s''\Psi_q(1+Z^2)+(s')^2B_s+\mathcal E_s,
 \label{eq:barrier-master-increment}
\end{align}
and, whenever a fixed \(M\) satisfies
\begin{equation}\label{eq:s-assumptions-direct}
 |s|\leq Mq|\log q|,\qquad
 |s'|\leq Mq|s|,\qquad
 |s''|\leq Mq^2|s|,
\end{equation}
the remainder obeys
\begin{equation}\label{eq:barrier-error-increment}
 |\mathcal E_s|
 \leq C_Mq^{1/4}\bigl(|s|S+s^2X^2\bigr).
\end{equation}
We shall also prove directly that
\begin{equation}\label{eq:residual-direct}
 |\Fop[V_q]|\leq C_Rq^3S.
\end{equation}

Assume these three formulas for the moment.  For the lower barrier,
take \(s=-a_q\).  Then \(s'=s''=0\), and
\begin{align*}
 \Fop[V_q-a_q\Psi_q]
 &\geq
 -C_Rq^3S+a_qS-a_q^2X^2
 -C_Kq^{1/4}(a_qS+a_q^2X^2)\\
 &=a_qS\left[
 1-\frac{C_R}{K}-a_qy^2S
 -C_Kq^{1/4}(1+a_qy^2S)
 \right].
\end{align*}
On the channel,
\[
 a_qy^2S\leq CKq.
\]
Choose \(K\) so large that \(C_R/K\leq1/8\).  With \(K\) fixed,
decrease \(q\) until
\[
 CKq\leq\frac18,\qquad
 C_Kq^{1/4}(1+CKq)\leq\frac18.
\]
It follows that
\begin{equation}\label{eq:lower-F-positive}
 \Fop[V_q-a_q\Psi_q]\geq\frac58a_qS>0.
\end{equation}

For the upper barrier, set
\[
 s_q=a_q+e_q.
\]
Direct differentiation gives
\[
 |e_q'|\leq\varepsilon_0qe_q,\qquad
 e_q''=\varepsilon_0^2q^2e_q.
\]
Since \(s_q\geq0\),
\begin{equation}\label{eq:upper-ZB-direct}
 1+Z^2\leq C(S+s_q^2X^2),\qquad
 (B_{s_q})_+\leq Cq^{-2}X^2.
\end{equation}
Indeed, \(yZ=yt-s_qy^2S\leq yt\), and
\[
 S(1-s_q(1+2yt))-1\leq S-1=t^2.
\]
Using \(0\leq\Psi_q\leq Cq^{-1}\) and
\(|t|\leq X\) in \eqref{eq:Bs-direct} gives the second estimate.
At \(y=0\), one has \(B_{s_q}=-s_q\Psi_q^2\leq0\), so the estimate
also covers the central line.

It follows that
\begin{align}
 e_q''\Psi_q(1+Z^2)
 &\leq C\varepsilon_0^2qe_q(S+s_q^2X^2),
 \label{eq:epp-direct}\\
 (e_q')^2(B_{s_q})_+
 &\leq C\varepsilon_0^2s_q^2X^2.
 \label{eq:ep-direct}
\end{align}
Put
\[
 A=s_qS,\qquad B=s_q^2X^2.
\]
Since \(s_q\geq a_q=Kq^3\) and \(e_q\leq s_q\),
\[
 C_Rq^3S\leq\frac{C_R}{K}A,\qquad
 qe_qS\leq qA,\qquad
 qe_q\,s_q^2X^2=(qe_q)B.
\]
Equations \eqref{eq:barrier-master-increment}--
\eqref{eq:ep-direct} therefore yield
\begin{align}
 \Fop[V_q+s_q\Psi_q]
 \leq{}&
 \left[-1+\frac{C_R}{K}+C_Mq^{1/4}
 +C\varepsilon_0^2q\right]A\notag\\
 &+
 \left[-1+C_Mq^{1/4}+C\varepsilon_0^2
 +C\varepsilon_0^2qe_q\right]B.
 \label{eq:upper-sign-reduction}
\end{align}
After \(K\) has been fixed, choose \(\varepsilon_*>0\) so that
\(C\varepsilon_*^2\leq1/8\).  The already fixed
\(0<\varepsilon_0\leq\varepsilon_*\) then satisfies the same
inequality.  Fix \(C_E\); its size will be chosen later for the
boundary comparison, and choose
\(M=M(K,\varepsilon_0,C_E)\) so that \(s_q\) satisfies
\eqref{eq:s-assumptions-direct}.  Now decrease \(q\) until
\[
 C_Mq^{1/4}\leq\frac18,\qquad
 C\varepsilon_0^2q\leq\frac18,\qquad
 C\varepsilon_0^2qe_q\leq\frac18.
\]
The last condition is possible because
\[
 0\leq qe_q\leq C_Eq^2|\log q|\longrightarrow0.
\]
Both coefficients in \eqref{eq:upper-sign-reduction} are then
negative, and hence
\begin{equation}\label{eq:upper-F-negative}
 \Fop[V_q+s_q\Psi_q]\leq0.
\end{equation}

It remains to establish
\eqref{eq:barrier-master-increment}--\eqref{eq:residual-direct} from the
definitions.  This is done in the next three steps.

\medskip
\noindent\emph{Step 1: the residual of the corrected profile.}
From \eqref{eq:parameters} and \eqref{eq:gamma},
\begin{align*}
 |\gamma_q(x)-\gamma_0|
 &\leq\lambda\Lambda_q+q^7\Lambda_q^2\\
 &\leq\Theta q+O(q^{9/4})+O(q^{9/2}).
\end{align*}
Since
\[
 \frac{\pi}{2}-\gamma_0=\arcsin q=q+O(q^3)
\]
and \(\Theta\) was fixed sufficiently small, it follows that
\begin{equation}\label{eq:channel-scales-direct}
 cq\leq\frac{\pi}{2}-\gamma_q(x)\leq Cq,\qquad
 |T|\leq Cq^{-1},\qquad
 \Sigma\leq Cq^{-2}.
\end{equation}
Because \(|y|\leq\gamma_q(x)\),
\begin{equation}\label{eq:X-scales-direct}
 S\leq Cq^{-2},\qquad
 |t|\leq X\leq CS,\qquad
 y^2S\leq Cq^{-2},\qquad
 q^2X^2\leq CS.
\end{equation}
Here \(|t|\leq X\) follows by integrating \(\sec^2r\) from \(0\) to
\(y\), and the last inequality follows from \(q^2S\leq C\).
Moreover,
\begin{equation}\label{eq:gamma-scales-direct}
 \gamma_q'=q^{9/4}-2q^7x,\qquad
 \gamma_q''=-2q^7,\qquad
 |(\gamma_q')^2-\tau|\leq Cq^8.
\end{equation}

We next differentiate the corrector in \eqref{eq:wdef}.  One obtains
\begin{align}
 w_y&=yS+\frac{T^2}{2}(t+yS),\label{eq:wy-direct}\\
 w_{yy}&=S(1+2yt)+T^2S(1+yt),\label{eq:wyy-direct}\\
 w_{\gamma y}&=T\Sigma(t+yS),\label{eq:wgy-direct}\\
 w_\gamma&=-\gamma\Sigma
 +T\Sigma(yt-\gamma T)
 -\frac{T^2}{2}(T+\gamma\Sigma),\label{eq:wg-direct}\\
 w_{\gamma\gamma}
 &=-\Sigma-2\gamma T\Sigma
 +\Sigma(1+3T^2)(yt-\gamma T)\notag\\
 &\quad
 -2T\Sigma(T+\gamma\Sigma)
 -T^2\Sigma(1+\gamma T).
 \label{eq:wgg-direct}
\end{align}
Equations \eqref{eq:channel-scales-direct}--
\eqref{eq:X-scales-direct} give
\begin{align}
 |w_y|&\leq Cq^{-2}X,
 &|w_{yy}|&\leq Cq^{-3}S,\notag\\
 |w_{\gamma y}|&\leq Cq^{-3}X,
 &|w_\gamma|&\leq Cq^{-4},
 &|w_{\gamma\gamma}|&\leq Cq^{-5}.
 \label{eq:w-bounds-direct}
\end{align}
We shall also use the more balanced consequence of
\eqref{eq:wyy-direct},
\begin{equation}\label{eq:wyy-balanced}
 |w_{yy}|\leq Cq^{-2}(S+X^2).
\end{equation}

Set
\[
 G(\gamma,y)=g(y)-g(\gamma)+\tau w(\gamma,y),
 \qquad V_q(x,y)=G(\gamma_q(x),y).
\]
The chain rule gives
\[
 (V_q)_x=\gamma_q'G_\gamma,\quad
 (V_q)_y=G_y,\quad
 (V_q)_{xx}=(\gamma_q')^2G_{\gamma\gamma}
 +\gamma_q''G_\gamma,
\]
\[
 (V_q)_{xy}=\gamma_q'G_{\gamma y},\qquad
 (V_q)_{yy}=G_{yy}.
\]
Substitution in \eqref{eq:Fdef} yields the exact identity
\begin{align}
 \Fop[V_q]
 ={}&G_{yy}-1-G_y^2\notag\\
 &+(\gamma_q')^2
 \bigl[(1+G_y^2)G_{\gamma\gamma}
 -2G_\gamma G_yG_{\gamma y}
 +G_\gamma^2(G_{yy}-1)\bigr]\notag\\
 &+\gamma_q''(1+G_y^2)G_\gamma.
 \label{eq:F-chain-direct}
\end{align}
Differentiating with respect to $G$ gives:
\begin{equation}\label{eq:G-derivatives-summary}
\begin{aligned}
G_\gamma
   &= -T+\tau w_\gamma,
&\qquad
G_y
   &= t+\tau w_y,\\
G_{\gamma\gamma}
   &= -\Sigma+\tau w_{\gamma\gamma},
&
G_{\gamma y}=G_{y\gamma}
   &= \tau w_{\gamma y},\\
G_{yy}
   &= S+\tau w_{yy}.
\end{aligned}
\end{equation}
First replace \((\gamma_q')^2\) by \(\tau\) and
\(\gamma_q''\) by zero.  The coefficient of \(\tau\) in
\eqref{eq:F-chain-direct} is
\[
 w_{yy}-2tw_y-\Sigma S+T^2t^2.
\]
The explicit formulas above give
\[
 w_{yy}-2tw_y=S+T^2,
\]
while
\[
 \Sigma S-T^2t^2
 =(1+T^2)S-T^2t^2=S+T^2.
\]
Thus the complete term linear in \(\tau\) vanishes.

For brevity in the remaining exact expansion, write
\[
 \alpha=w_\gamma,\quad \beta=w_y,\quad
 A_1=w_{\gamma\gamma},\quad
 C_1=w_{\gamma y},\quad D_1=w_{yy}.
\]

The affine-width residual is
\begin{equation}\label{eq:F-affine-direct}
 \tau^2R_2+\tau^3R_3+\tau^4R_4,
\end{equation}
where
\begin{align}
 R_2={}&
 SA_1-2\Sigma t\beta+2TtC_1+T^2D_1
 -2Tt^2\alpha-\beta^2,\label{eq:R2-direct}\\
 R_3={}&
 -\Sigma\beta^2+2A_1t\beta+2TC_1\beta
 -2tC_1\alpha-2TD_1\alpha+t^2\alpha^2,
 \label{eq:R3-direct}\\
 R_4={}&A_1\beta^2-2C_1\alpha\beta+D_1\alpha^2.
 \label{eq:R4-direct}
\end{align}
Using \eqref{eq:X-scales-direct} and
\eqref{eq:w-bounds-direct} term by term, we obtain
\[
 |R_2|\leq Cq^{-6}S,\qquad
 |R_3|\leq Cq^{-8}S,\qquad
 |R_4|\leq Cq^{-11}S.
\]
For example,
\[
 |\Sigma t\beta|
 \leq Cq^{-4}|t|X\leq Cq^{-5}S,
\qquad
\beta^2 \leq C q^{-4} |y|^2 S^2\leq C q^{-6} S,
\]
and
\[
 |A_1t\beta|
 \leq Cq^{-7}|t|X\leq Cq^{-8}S,
 \qquad 
 |A_1\beta^2|
 \leq Cq^{-9}X^2\leq Cq^{-11}S.
\]
Since \(\tau=q^{9/2}\), \eqref{eq:F-affine-direct} is bounded by
\[
 C\bigl(q^9q^{-6}+q^{27/2}q^{-8}
 +q^{18}q^{-11}\bigr)S
 \leq Cq^3S.
\]

It remains to restore the true width.  Put
\[
 \mathcal B
 =(1+G_y^2)G_{\gamma\gamma}
 -2G_\gamma G_yG_{\gamma y}
 +G_\gamma^2(G_{yy}-1).
\]
The preceding estimates imply
\begin{align*}
 |G_\gamma|&\leq Cq^{-1},
 &|G_{\gamma\gamma}|&\leq Cq^{-2},
 &|G_{\gamma y}|&\leq Cq^{3/2}X,\\
 |G_{yy}|&\leq CS,
 &1+G_y^2&\leq CS.
\end{align*}
Consequently,
\[
 |\mathcal B|\leq Cq^{-2}S,\qquad
 |(1+G_y^2)G_\gamma|\leq Cq^{-1}S.
\]
By \eqref{eq:gamma-scales-direct}, the error made in the affine-width
calculation is at most
\[
 |(\gamma_q')^2-\tau|\,|\mathcal B|
 +|\gamma_q''|\,|(1+G_y^2)G_\gamma|
 \leq Cq^6S.
\]
This proves \eqref{eq:residual-direct}.

\medskip
\noindent\emph{Step 2: the full perturbation, including the
\(x\)-dependence of \(\Psi_q\).}
Since \(\Acal'(r)=r\sec^2r\), direct differentiation of
\eqref{eq:Psi-def} gives
\begin{align}
 (\Psi_q)_y&=-yS,
 &(\Psi_q)_{yy}&=-S(1+2yt),\notag\\
 (\Psi_q)_x&=\gamma_q'\gamma_q\Sigma,
 &(\Psi_q)_{xx}
 &=\gamma_q''\gamma_q\Sigma + \notag(\gamma_q')^2\Sigma
   \bigl(1+2\gamma_q T\bigr).
 \label{eq:Psi-derivatives-direct}
\end{align}
Therefore
\begin{equation}\label{eq:Psi-bounds-direct}
 0\leq\Psi_q\leq Cq^{-1},\qquad
 |(\Psi_q)_x|\leq Cq^{1/4},\qquad
 |(\Psi_q)_{xx}|\leq Cq^{3/2}.
\end{equation}
Indeed, the two contributions to \((\Psi_q)_{xx}\) have sizes
\(O(q^5)\) and \(O(q^{3/2})\), respectively.

For arbitrary smooth functions \(u\) and \(\phi\), direct expansion of
\eqref{eq:Fdef} gives
\begin{align}
 \Fop[u+\phi]-\Fop[u]
 ={}&\phi_{xx}\bigl(1+(u_y+\phi_y)^2\bigr)
 -2\phi_{xy}(u_x+\phi_x)(u_y+\phi_y)\notag\\
 &+\phi_{yy}\bigl(1+(u_x+\phi_x)^2\bigr)\notag\\
 &+(2u_x\phi_x+\phi_x^2)(u_{yy}-1)
 +(2u_y\phi_y+\phi_y^2)(u_{xx}-1)\notag\\
 &-2(u_x\phi_y+u_y\phi_x+\phi_x\phi_y)u_{xy}.
 \label{eq:full-increment-direct}
\end{align}

We apply this identity with \(u=V_q\) and
\(\phi=s\Psi_q\), where \(s=s(x)\).  The required derivatives are
\begin{align*}
 (s\Psi_q)_x&=s'\Psi_q+s(\Psi_q)_x,
 & (s\Psi_q)_y&=-syS,\\
 (s\Psi_q)_{xx}
 &=s''\Psi_q+2s'(\Psi_q)_x+s(\Psi_q)_{xx},
 & (s\Psi_q)_{xy}&=-s'yS,\\
 (s\Psi_q)_{yy}&=-sS(1+2yt).
\end{align*}
Recall that
\[
 Z=t-syS,\qquad X=|y|S,
\]
and
\[
 B_s
 =2\Psi_qZyS
 +\Psi_q^2\bigl[S(1-s(1+2yt))-1\bigr].
\]

To identify the main terms, freeze the derivatives of \(V_q\) at the
grim-reaper jet
\[
 u_x=0,\qquad u_y=t,\qquad
 u_{xx}=u_{xy}=0,\qquad u_{yy}=S,
\]
and retain \(s'\Psi_q\) and \(s''\Psi_q\) as the corresponding
\(x\)-derivatives of \(s\Psi_q\).  The resulting contribution is
\begin{align*}
 &s''\Psi_q(1+Z^2)
 +2(s')^2\Psi_qZyS
 -sS(1+2yt)\bigl[1+(s'\Psi_q)^2\bigr]\\
 &\qquad
 +(s')^2\Psi_q^2(S-1)+2st\,yS-s^2X^2\\
 &=
 s''\Psi_q(1+Z^2)-sS-s^2X^2+(s')^2B_s.
\end{align*}
Thus the signed term \(-sS-s^2X^2\) and the quadratic term
\((s')^2B_s\) arise directly from the frozen grim-reaper calculation.

Adding and subtracting this frozen contribution in
\eqref{eq:full-increment-direct} gives the exact identity
\begin{equation}\label{eq:barrier-master-direct}
 \Fop[V_q+s\Psi_q]
 =\Fop[V_q]-sS-s^2X^2
 +s''\Psi_q(1+Z^2)+(s')^2B_s+\mathcal E_s,
\end{equation}
where
\begin{align}
 \mathcal E_s={}&
 \Bigl\{ \bigl(2s'(\Psi_q)_x+s(\Psi_q)_{xx}\bigr)
 \Bigl(1+\bigl[Z+(V_q)_y-t\bigr]^2\Bigr)\notag\\
 &+s''\Psi_q\bigl((V_q)_y-t\bigr)
 \bigl(2Z+(V_q)_y-t\bigr) \Bigr\}\notag\\[1mm]
 &+2s'yS\Bigl\{
 (V_q)_x\bigl[Z+(V_q)_y-t\bigr]
 +s(\Psi_q)_xZ\notag\\
 &\hspace{31mm}
 +\bigl(s'\Psi_q+s(\Psi_q)_x\bigr)
   \bigl((V_q)_y-t\bigr)
 \Bigr\}\notag\\[1mm]
 &-sS(1+2yt)
 \Bigl\{
 \bigl[(V_q)_x+s'\Psi_q+s(\Psi_q)_x\bigr]^2
 -(s'\Psi_q)^2
 \Bigr\}\notag\\[1mm]
 &+(S-1)\Bigl\{
 2(V_q)_x\bigl(s'\Psi_q+s(\Psi_q)_x\bigr)
 +\bigl(s'\Psi_q+s(\Psi_q)_x\bigr)^2
 -(s'\Psi_q)^2
 \Bigr\}\notag\\
 &+\bigl((V_q)_{yy}-S\bigr)
 \Bigl\{
 2(V_q)_x\bigl(s'\Psi_q+s(\Psi_q)_x\bigr)
 +\bigl(s'\Psi_q+s(\Psi_q)_x\bigr)^2
 \Bigr\}\notag\\[1mm]
 &+(V_q)_{xx}
 \bigl[-2syS(V_q)_y+s^2X^2\bigr]
 +2syS\bigl((V_q)_y-t\bigr)\notag\\[1mm]
 &-2(V_q)_{xy}\Bigl\{
 -syS(V_q)_x
 +(V_q)_y\bigl(s'\Psi_q+s(\Psi_q)_x\bigr)\notag\\
 &\hspace{34mm}
 -syS\bigl(s'\Psi_q+s(\Psi_q)_x\bigr)
 \Bigr\}.
 \label{eq:E-exact-direct}
\end{align}
In particular, all terms involving \((\Psi_q)_x\) or
\((\Psi_q)_{xx}\) are contained explicitly in \(\mathcal E_s\).

We now estimate this remainder.  Step~1 gives
\begin{align}
 |(V_q)_x|&\leq Cq^{5/4},
 &|(V_q)_y-t|&\leq Cq^{5/2}X,\notag\\
 |(V_q)_{xx}|&\leq Cq^{5/2},
 &|(V_q)_{xy}|&\leq Cq^{15/4}X,\notag\\
 |(V_q)_{yy}-S|&\leq Cq^{5/2}(S+X^2).
 \label{eq:base-jet-direct}
\end{align}
Under \eqref{eq:s-assumptions-direct},
\eqref{eq:Psi-bounds-direct} implies
\begin{align}
 |s'\Psi_q|&\leq C_M|s|,
 &|s(\Psi_q)_x|&\leq Cq^{1/4}|s|,\notag\\
 |s''\Psi_q|&\leq C_Mq|s|,
 &|2s'(\Psi_q)_x+s(\Psi_q)_{xx}|
     &\leq C_Mq^{5/4}|s|,\notag\\
 |s'yS|&\leq C_Mq|s|X,
 &|syS|&=|s|X,\notag\\
 |sS(1+2yt)|&\leq C|s|(S+X^2).
 \label{eq:perturbation-jet-direct}
\end{align}
We also have
\begin{align}
 |s'\Psi_q+s(\Psi_q)_x|
 &\leq C_M|s|,\notag\\
 \left|
 \bigl(s'\Psi_q+s(\Psi_q)_x\bigr)^2-(s'\Psi_q)^2
 \right|
 &\leq C_Mq^{1/4}s^2.
 \label{eq:perturbation-square-direct}
\end{align}
Finally,
\[
 |t|\leq X,\qquad t^2\leq S,\qquad
 X^2\leq Cq^{-2}S,
\]
and hence
\[
 1+\bigl[Z+(V_q)_y-t\bigr]^2
 \leq C(S+s^2X^2).
\]

We now record the substitution term by term.  Write
\[
 \mathfrak A=|s|S,\qquad \mathfrak B=s^2X^2,
\]
and decompose \(\mathcal E_s=\sum_{j=1}^8E_j\) according to the eight
successive displayed groups in \eqref{eq:E-exact-direct}: the two
terms in the first braces form \(E_1\), the next braces form \(E_2\),
the following three product lines form \(E_3,E_4,E_5\), the
\((V_q)_{xx}\)-term and the subsequent
\(2syS((V_q)_y-t)\)-term form \(E_6,E_7\), and the final
\((V_q)_{xy}\)-term is \(E_8\).  Equations
\eqref{eq:base-jet-direct}--\eqref{eq:perturbation-square-direct}
give
\begin{equation}\label{eq:E-groups-direct}
 \begin{array}{c|c}
  \text{group}&\text{upper bound}\\ \hline
  E_1&C_Mq^{5/4}(\mathfrak A+\mathfrak B)\\
  E_2&C_Mq^{1/2}(\mathfrak A+\mathfrak B)\\
  E_3&C_Mq^{1/2}(\mathfrak A+\mathfrak B)\\
  E_4&C_Mq^{1/2}(\mathfrak A+\mathfrak B)\\
  E_5&C_Mq^{1/2}(\mathfrak A+\mathfrak B)\\
  E_6&C_Mq^{1/2}(\mathfrak A+\mathfrak B)\\
  E_7&C_Mq^{1/2}\mathfrak A\\
  E_8&C_Mq^{1/2}(\mathfrak A+\mathfrak B)
 \end{array}
\end{equation}
for all sufficiently small \(q\), with the smallness threshold
depending only on \(M\).  Here are the details.  Set
\[
 D=(V_q)_y-t,\qquad
 P=s'\Psi_q+s(\Psi_q)_x,\qquad
 Q=2s'(\Psi_q)_x+s(\Psi_q)_{xx}.
\]
The preceding estimates imply
\[
 |D|\leq Cq^{5/2}X,\quad
 |P|\leq C_M|s|,\quad
 |Q|\leq C_Mq^{5/4}|s|,\quad
 |P^2-(s'\Psi_q)^2|\leq C_Mq^{1/4}s^2,
\]
as well as
\[
 |Z|\leq S^{1/2}+|s|X,\qquad
 |(V_q)_y|\leq CX.
\]
Using these inequalities in the first two groups gives
\begin{align*}
 |E_1|
 &\leq C_Mq^{5/4}|s|(S+\mathfrak B)
   +C_Mq|s|\,q^{5/2}X
    \bigl(S^{1/2}+|s|X+q^{5/2}X\bigr)\\
 &\leq C_Mq^{5/4}(\mathfrak A+\mathfrak B),\\
 |E_2|
 &\leq C_Mq|s|X\Bigl[
   q^{5/4}\bigl(S^{1/2}+|s|X+q^{5/2}X\bigr)\\
 &\hspace{34mm}
   +q^{1/4}|s|\bigl(S^{1/2}+|s|X\bigr)
   +q^{5/2}|s|X\Bigr]\\
 &\leq C_Mq^{1/2}(\mathfrak A+\mathfrak B).
\end{align*}
For the remaining six groups, direct substitution yields
\begin{align*}
 |E_3|
 &\leq C|s|(S+X^2)
   \bigl(q^{5/2}+q^{5/4}|s|+q^{1/4}s^2\bigr),\\
 |E_4|
 &\leq CS\bigl(q^{5/4}|s|+q^{1/4}s^2\bigr),\\
 |E_5|
 &\leq C_Mq^{5/2}(S+X^2)
   \bigl(q^{5/4}|s|+s^2\bigr),\\
 |E_6|
 &\leq Cq^{5/2}\bigl(|s|X^2+\mathfrak B\bigr),\\
 |E_7|
 &\leq Cq^{5/2}|s|X^2,\\
 |E_8|
 &\leq C_Mq^{15/4}X
   \bigl(q^{5/4}|s|X+|s|X+s^2X\bigr).
\end{align*}
Since \(X^2\leq Cq^{-2}S\) and
\(|s|\leq Mq|\log q|\), one has, for sufficiently small \(q\),
\[
 q^{5/2}|s|X^2\leq Cq^{1/2}\mathfrak A,
 \qquad
 q^{1/4}s^2S\leq C_Mq^{1/2}\mathfrak A,
 \qquad
 |s|\leq q^{1/4}.
\]
These three absorptions applied to the last display prove the bounds
for \(E_3,\ldots,E_8\) in \eqref{eq:E-groups-direct}.  Thus every
entry of that table has been verified.  Since
\(q^{1/2}\leq q^{1/4}\) for
\(0<q\leq1\), summing the eight groups gives
\begin{equation}\label{eq:barrier-error-direct}
 |\mathcal E_s|
 \leq C_Mq^{1/4}\bigl(|s|S+s^2X^2\bigr).
\end{equation}

\medskip
\noindent\emph{Step 3: conclusion.}
The functions \(s=-a_q\) and \(s=s_q=a_q+e_q\) satisfy
\eqref{eq:s-assumptions-direct} for a fixed
\(M=M(K,\varepsilon_0,C_E)\).  Thus Steps 1 and 2 justify all
estimates used to obtain \eqref{eq:lower-F-positive} and
\eqref{eq:upper-F-negative}.  Finally, \eqref{eq:FQ} shows that
\(\Fop[u]\) and \(\Qop[u]-1\) have the same sign.  The two inequalities
in \eqref{eq:direct-barriers} follow.
\end{proof}

\subsection{Boundary comparison and the sandwich estimate}

\begin{theorem}[Quantitative comparison with \(V_q\)]
\label{thm:quantitative-comparison}
Choose the constants in the order
\[
 K,\qquad \varepsilon_0,\qquad C_E,
\]
where \(K\) and \(\varepsilon_0\) are chosen as in
Theorem~\ref{thm:direct-barriers}, and \(C_E\) is then chosen
sufficiently large.  For all sufficiently small \(q\), the solution
\(u_q\) of \eqref{eq:uq-problem} satisfies
\begin{equation}\label{eq:sandwich}
 V_q-a_q\Psi_q
 \leq u_q
 \leq V_q+(a_q+e_q)\Psi_q
 \qquad\text{in }\Om_q.
\end{equation}
Moreover, for every fixed \(0<\eta<1\),
\begin{equation}\label{eq:central-comparison}
 |u_q-V_q|
 \leq Cq^2+
 C|\log q|\exp\!\bigl(-c(1-\eta)q^{-1/4}\bigr)
\end{equation}
on \(\Om_q\cap\{|x|\leq\eta L_q\}\), where \(c,C>0\) are
independent of \(q\).
\end{theorem}

\begin{proof}
\medskip
\noindent\emph{Step 1: a uniform estimate for \((-V_q)/\Psi_q\).}
We first claim that there is a constant \(C_*>0\), independent of
\(q\), such that
\begin{equation}\label{eq:ratio-direct}
 0\leq\frac{-V_q}{\Psi_q}
 \leq C_*q|\log q|
\end{equation}
throughout the interior of \(\mathcal C_q\).

Since \(g\), \(\Acal\), and \(w(\gamma,\cdot)\) are even, we may put
\[
 \gamma=\gamma_q(x),\qquad
 r=|y|,\qquad
 d=\gamma-r,\qquad
 T=\tan\gamma.
\]
By \eqref{eq:channel-scales-direct},
\[
 \frac{\pi}{2}-\gamma\asymp q,
 \qquad
 \frac{\pi}{2}-r
 =\frac{\pi}{2}-\gamma+d\asymp q+d.
\]

We first estimate the grim-reaper contribution
\[
 \frac{g(\gamma)-g(r)}
 {\Acal(\gamma)-\Acal(r)}.
\]
Suppose first that \(d\leq q\).  Every \(s\in[r,\gamma]\) then
satisfies
\[
 \frac{\pi}{2}-s\asymp q.
\]
Consequently,
\[
 \tan s\leq \frac{C}{q},
 \qquad
 s\sec^2s\geq\frac{c}{q^2}
 \qquad (r\leq s\leq\gamma),
\]
and hence
\[
 g(\gamma)-g(r)
 =\int_r^\gamma\tan s\,\dd s
 \leq\frac{Cd}{q},
\]
whereas
\[
 \Acal(\gamma)-\Acal(r)
 =\int_r^\gamma s\sec^2s\,\dd s
 \geq\frac{cd}{q^2}.
\]
Thus
\[
 \frac{g(\gamma)-g(r)}
 {\Acal(\gamma)-\Acal(r)}
 \leq Cq.
\]

Suppose next that \(d\geq q\) and \(r\geq\gamma/2\).  In this range,
\[
 \cos\gamma\asymp q,
 \qquad
 \cos r\asymp q+d,
\]
so that
\[
 \begin{aligned}
 g(\gamma)-g(r)
 &=\log\frac{\cos r}{\cos\gamma}\\
 &\leq C\log\left(1+\frac{d}{q}\right).
 \end{aligned}
\]
Moreover, \(s\geq\gamma/2\geq c>0\) on \([r,\gamma]\), and therefore
\[
 \begin{aligned}
 \Acal(\gamma)-\Acal(r)
 &=\int_r^\gamma s\sec^2s\,\dd s\\
 &\geq c\int_r^\gamma\sec^2s\,\dd s\\
 &=c(\tan\gamma-\tan r)\\
 &=c\frac{\sin d}{\cos\gamma\cos r}
 \geq\frac{cd}{q(q+d)}.
 \end{aligned}
\]
It follows that
\[
 \begin{aligned}
 \frac{g(\gamma)-g(r)}
 {\Acal(\gamma)-\Acal(r)}
 &\leq
 Cq\left(1+\frac{q}{d}\right)
 \log\left(1+\frac{d}{q}\right)\\
 &\leq Cq|\log q|,
 \end{aligned}
\]
where we used \(d\geq q\) and \(d\leq\gamma<\pi/2\).

Finally, suppose that \(r\leq\gamma/2\).  Since \(\Acal\) is
increasing on \([0,\pi/2)\),
\[
 \begin{aligned}
 \Acal(\gamma)-\Acal(r)
 &\geq\Acal(\gamma)-\Acal(\gamma/2)\\
 &=\int_{\gamma/2}^{\gamma}s\sec^2s\,\dd s\\
 &\geq c\bigl(\tan\gamma-\tan(\gamma/2)\bigr)
 \geq\frac{c}{q}.
 \end{aligned}
\]
On the other hand,
\[
 g(\gamma)-g(r)
 \leq g(\gamma)-g(0)
 =-\log\cos\gamma
 \leq C|\log q|.
\]
Thus the grim-reaper contribution is again bounded by
\(Cq|\log q|\).

We now estimate the corrector.  Since
\[
 \frac{\dfrac{\dd}{\dd s}(s\tan s)}
 {\Acal'(s)}
 =
 1+\frac{\tan s}{s\sec^2s}
 \leq2,
\]
integration from \(r\) to \(\gamma\) gives
\[
 0\leq
 \frac{\gamma\tan\gamma-r\tan r}
 {\Acal(\gamma)-\Acal(r)}
 \leq2.
\]
The explicit formula \eqref{eq:wdef} therefore yields
\[
 \begin{aligned}
 \frac{-w(\gamma,r)}
 {\Acal(\gamma)-\Acal(r)}
 &=
 1+\frac{T^2}{2}
 \frac{\gamma T-r\tan r}
 {\Acal(\gamma)-\Acal(r)}\\
 &\leq1+T^2.
 \end{aligned}
\]
Since \(\tau=q^{9/2}\) and \(T^2\leq Cq^{-2}\),
\[
 \tau\frac{-w(\gamma,r)}
 {\Acal(\gamma)-\Acal(r)}
 \leq Cq^{9/2}(1+q^{-2})
 \leq Cq^{5/2}.
\]
Combining this with the preceding grim-reaper estimates and the
identity
\[
 \frac{-V_q}{\Psi_q}
 =
 \frac{g(\gamma)-g(r)}
 {\Acal(\gamma)-\Acal(r)}
 +
 \tau\frac{-w(\gamma,r)}
 {\Acal(\gamma)-\Acal(r)},
\]
we obtain \eqref{eq:ratio-direct}.  Equivalently,
\[
 0\leq -V_q
 \leq C_*q|\log q|\,\Psi_q
\]
in the interior of \(\mathcal C_q\), and this last inequality extends
to its closure by continuity.

\medskip
\noindent\emph{Step 2: boundary ordering and comparison.}
Consider first a point on a cap or interpolation part of
\(\partial\Om_q\).  By \eqref{eq:cap-localization},
\[
 0\leq\Lambda_q-|x|\leq C_0.
\]
Using the evenness of \(\cosh\), we obtain
\[
 \begin{aligned}
 \frac{\cosh(\varepsilon_0qx)}
 {\cosh(\varepsilon_0q\Lambda_q)}
 &=
 \frac{\cosh(\varepsilon_0q|x|)}
 {\cosh(\varepsilon_0q\Lambda_q)}\\
 &\geq
 \exp\!\bigl(-\varepsilon_0q(\Lambda_q-|x|)\bigr)\\
 &\geq e^{-\varepsilon_0C_0}.
 \end{aligned}
\]
Consequently,
\[
 e_q(x)
 \geq C_Ee^{-\varepsilon_0C_0}q|\log q|
\]
on every cap and interpolation arc.  We now choose
\[
 C_E\geq2e^{\varepsilon_0C_0}C_*.
\]
At every such boundary point for which \(\Psi_q>0\),
\[
 \begin{aligned}
 V_q+(a_q+e_q)\Psi_q
 &=
 \Psi_q\left(
 a_q+e_q-\frac{-V_q}{\Psi_q}
 \right)\\
 &\geq
 \Psi_q\bigl(a_q+C_*q|\log q|\bigr)
 \geq0.
 \end{aligned}
\]
On each unchanged lateral arc,
\[
 y=\pm\gamma_q(x),
\]
we have \(V_q=\Psi_q=0\), so equality holds.  The same conclusion
holds at the junction points by continuity.  We have therefore proved
\begin{equation}\label{eq:upper-boundary-direct}
 V_q+(a_q+e_q)\Psi_q\geq0
 \qquad\text{on }\partial\Om_q.
\end{equation}

For the lower barrier, no localization is needed.  Since
\(\Om_q\subset\mathcal C_q\), \(V_q\leq0\), and \(\Psi_q\geq0\),
\begin{equation}\label{eq:lower-boundary-direct}
 V_q-a_q\Psi_q\leq0
 \qquad\text{on }\partial\Om_q.
\end{equation}

We now apply the comparison principle twice.  By
Theorem~\ref{thm:direct-barriers},
\[
 \Qop[V_q-a_q\Psi_q]\geq1=\Qop[u_q]
 \qquad\text{in }\Om_q,
\]
while \eqref{eq:lower-boundary-direct} and \(u_q=0\) give
\[
 V_q-a_q\Psi_q\leq u_q
 \qquad\text{on }\partial\Om_q.
\]
Lemma~\ref{lem:comparison} therefore gives
\[
 V_q-a_q\Psi_q\leq u_q
 \qquad\text{in }\Om_q.
\]

Similarly,
\[
 \Qop[u_q]=1
 \geq\Qop[V_q+(a_q+e_q)\Psi_q]
 \qquad\text{in }\Om_q,
\]
and \eqref{eq:upper-boundary-direct} gives
\[
 u_q\leq V_q+(a_q+e_q)\Psi_q
 \qquad\text{on }\partial\Om_q.
\]
A second application of Lemma~\ref{lem:comparison} yields
\[
 u_q\leq V_q+(a_q+e_q)\Psi_q
 \qquad\text{in }\Om_q.
\]
This proves \eqref{eq:sandwich}.

\medskip
\noindent\emph{Step 3: the estimate in the central channel.}
Suppose that \(|x|\leq\eta L_q\).  Since
\[
 \frac{\cosh(\varepsilon_0qx)}
 {\cosh(\varepsilon_0q\Lambda_q)}
 \leq
 2\exp\!\bigl(-\varepsilon_0q(\Lambda_q-|x|)\bigr),
\]
and
\[
 \Lambda_q=L_q+6,
 \qquad
 qL_q=\Theta q^{-1/4},
\]
we have
\[
 \begin{aligned}
 q(\Lambda_q-|x|)
 &\geq q(\Lambda_q-\eta L_q)\\
 &=\Theta(1-\eta)q^{-1/4}+6q.
 \end{aligned}
\]
It follows that
\[
 e_q(x)
 \leq Cq|\log q|
 \exp\!\bigl(-c(1-\eta)q^{-1/4}\bigr).
\]

Moreover, the monotonicity of \(\Acal\) and
\eqref{eq:channel-scales-direct} imply
\[
 0\leq\Psi_q
 \leq\Acal(\gamma_q(x))
 \leq\gamma_q(x)\tan\gamma_q(x)
 \leq Cq^{-1}.
\]
Since \(a_q=Kq^3\), we obtain
\[
 a_q\Psi_q\leq Cq^2
\]
and
\[
 e_q\Psi_q
 \leq
 C|\log q|
 \exp\!\bigl(-c(1-\eta)q^{-1/4}\bigr).
\]
Finally, \eqref{eq:sandwich} gives
\[
 -a_q\Psi_q
 \leq u_q-V_q
 \leq(a_q+e_q)\Psi_q.
\]
Therefore
\[
 |u_q-V_q|
 \leq Cq^2+
 C|\log q|
 \exp\!\bigl(-c(1-\eta)q^{-1/4}\bigr)
\]
on \(\Om_q\cap\{|x|\leq\eta L_q\}\), which is
\eqref{eq:central-comparison}.
\end{proof}

\begin{corollary}[Stable transfer of a midpoint defect]
\label{cor:gap-transfer}
Fix \(0<\eta<1\), and suppose that
\[
 P_+,P_-,M\in\Om_q\cap\{|x|\leq\eta L_q\},
 \qquad M=\frac{P_++P_-}{2}.
\]
If, for some \(c_*\in\R\) and \(\delta>0\),
\[
 V_q(P_\pm)\leq c_*-\delta,\qquad
 V_q(M)\geq c_*+\delta,
\]
and if the right-hand side of \eqref{eq:central-comparison} is at most
\(\delta/2\), then
\[
 u_q(P_\pm)\leq c_*-\frac{\delta}{2},
 \qquad
 u_q(M)\geq c_*+\frac{\delta}{2}.
\]
In particular, the sublevel set \(\{u_q<c_*\}\) is not convex.
\end{corollary}

\begin{proof}
Apply \eqref{eq:central-comparison} at the three points.  The two
endpoints belong to the sublevel set, while their midpoint does not.
\end{proof}

\section{A nonconvex interior sublevel}
\label{sec:nonconvexity}
In this section, we complete the proof of the Theorem~\ref{thm:main}.

\begin{theorem}[A stable midpoint defect]\label{thm:stable-gap}
Let \(\Om_q\) and \(u_q\) be as in
Section~\ref{sec:domain}.  There are constants \(c_0,q_0>0\) such that,
for every \(0<q<q_0\), one can find
\[
 P_+,P_-,M_q\in\Om_q,
 \qquad
 M_q=\frac{P_++P_-}{2},
\]
for which
\begin{equation}\label{eq:uq-gap}
 u_q(P_\pm)\leq-\log 2-c_0q^{3/2},
 \qquad
 u_q(M_q)\geq-\log 2+c_0q^{3/2}.
\end{equation}
Consequently, \(\{u_q<-\log 2\}\) is not convex.  In fact,
\(\{u_q<c\}\) is not convex whenever
\begin{equation}\label{eq:nearby-levels}
 |c+\log 2|<\frac{c_0}{2}q^{3/2}.
\end{equation}
\end{theorem}

\begin{proof}
Set
\[
 c_*=-\log 2.
\]
The proof has two parts.  We first construct a quantitative midpoint
defect for the corrected profile \(V_q\).  We then use
Theorem~\ref{thm:quantitative-comparison} to transfer this defect to
the exact Dirichlet solution \(u_q\).

\medskip
\noindent\emph{Step 1: the upper root of the model level.}
For \(\gamma\) sufficiently close to \(\pi/2\), let
\(\mathscr R(\gamma,\tau)\in(0,\gamma)\) be the positive solution of
\begin{equation}\label{eq:root}
 g(\mathscr R)-g(\gamma)
 +\tau w(\gamma,\mathscr R)=c_*.
\end{equation}
We first justify this definition.  For fixed \(\gamma\), put
\[
 F_\gamma(R)=g(R)-g(\gamma)+\tau w(\gamma,R).
\]
Since \(w(\gamma,\gamma)=0\),
\[
 F_\gamma(\gamma)=0>c_*.
\]
On the other hand, \(g(0)=0\), \(g(\gamma)=-\log\cos\gamma\), and
the explicit formula \eqref{eq:wdef} gives \(w(\gamma,0)<0\).
Consequently,
\[
 F_\gamma(0)
 =\log\cos\gamma+\tau w(\gamma,0)
 <\log\cos\gamma<c_*
\]
when \(\gamma\) is sufficiently close to \(\pi/2\).  The intermediate
value theorem therefore gives a root in \((0,\gamma)\).  Moreover,
for \(0<R<\gamma\),
\[
 \partial_R F_\gamma(R)
 =\tan R+\tau w_y(\gamma,R)>0;
\]
indeed, all terms in the formula
\[
 w_y(\gamma,R)
 =R\sec^2R+\frac{\tan^2\gamma}{2}
 \bigl(\tan R+R\sec^2R\bigr)
\]
are positive.  Thus the root is unique.  The same strict inequality
and the implicit-function theorem show that
\(\mathscr R(\gamma,\tau)\) is smooth in the range considered below.

Fix \(0<a_0<1/8\), choose
\[
 0<\theta<\min\{\Theta/4,a_0/4\},
\]
and put
\begin{equation}\label{eq:Hq}
 H_q=\frac{\theta q}{\lambda}
 =\theta q^{-5/4},
 \qquad
 r_q(x)=\mathscr R(\gamma_q(x),\tau).
\end{equation}
It follows from \eqref{eq:Vdef} and \eqref{eq:root} that
\[
 V_q(x,r_q(x))=c_*
 \qquad (|x|\leq H_q).
\]

\medskip
\noindent\emph{Step 2: the curvature of the scaled root.}
We shall prove that
\begin{equation}\label{eq:Rgg-target}
 \mathscr R_{\gamma\gamma}
 \left(\frac{\pi}{2}-q\xi,q^{9/2}\right)
 =q^{1/2}\left(\frac{3\pi}{2\xi^4}+o(1)\right)
\end{equation}
uniformly for
\[
 \xi\in I:=[1-2a_0,1+2a_0].
\]
Introduce the scaled variables
\[
 \gamma=\frac{\pi}{2}-q\xi,
 \qquad
 \mathscr R=\frac{\pi}{2}-q\rho,
\]
and define
\[
 \Wcal(q,\xi,\rho)
 =q^3w\!\left(
 \frac{\pi}{2}-q\xi,
 \frac{\pi}{2}-q\rho
 \right).
\]
To identify the limit of \(\Wcal\), set
\[
 B_q(z)=\left(\frac{\pi}{2}-qz\right)\cot(qz).
\]
Since
\[
 \Acal\!\left(\frac{\pi}{2}-qz\right)
 =B_q(z)+\log\sin(qz),
\]
we have
\[
 \Acal(\mathscr R)-\Acal(\gamma)
 =B_q(\rho)-B_q(\xi)
 +\log\frac{\sin(q\rho)}{\sin(q\xi)}.
\]
Also,
\[
 \tan\gamma=\cot(q\xi),
 \qquad
 \mathscr R\tan\mathscr R-\gamma\tan\gamma
 =B_q(\rho)-B_q(\xi).
\]
Substitution in \eqref{eq:wdef}, followed by multiplication by
\(q^3\), gives the exact identity
\begin{align}
 \Wcal(q,\xi,\rho)
 ={}&q^3\left[
 B_q(\rho)-B_q(\xi)
 +\log\frac{\sin(q\rho)}{\sin(q\xi)}
 \right]\notag\\
 &+\frac12\bigl(q\cot(q\xi)\bigr)^2
 q\bigl[B_q(\rho)-B_q(\xi)\bigr].
 \label{eq:W-smooth-decomposition}
\end{align}

When \(z\) ranges in a fixed compact subset of \((0,\infty)\), the
functions
\[
 q\cot(qz),
 \qquad
 qB_q(z)=\left(\frac{\pi}{2}-qz\right)q\cot(qz),
 \qquad
 \frac{\sin(q\rho)}{\sin(q\xi)}
\]
extend smoothly across \(q=0\).  Their values at \(q=0\) are
\[
 \frac1z,
 \qquad
 \frac{\pi}{2z},
 \qquad
 \frac{\rho}{\xi},
\]
respectively.  The first line of
\eqref{eq:W-smooth-decomposition} tends to zero, while the second line
has a finite limit.  Hence \(\Wcal\) extends smoothly to \(q=0\), and
\begin{equation}\label{eq:Wlimit}
 \Wcal(0,\xi,\rho)
 =\frac{\pi}{4\xi^2}
 \left(\frac1\rho-\frac1\xi\right).
\end{equation}

The correct small parameter in the root equation is
\[
 \delta=\tau q^{-3}.
\]
Indeed, since \(g(s)=-\log\cos s\), equation \eqref{eq:root} is
equivalent to
\[
 \log\frac{2\cos\gamma}{\cos\mathscr R}
 +\tau w(\gamma,\mathscr R)=0.
\]
Using \(\cos\gamma=\sin(q\xi)\),
\(\cos\mathscr R=\sin(q\rho)\), and
\(\delta\Wcal=\tau w\), this becomes
\begin{equation}\label{eq:scaled-root}
 H(q,\xi,\rho,\delta)
 :=\log\frac{2\sin(q\xi)}{\sin(q\rho)}
 +\delta\Wcal(q,\xi,\rho)=0.
\end{equation}
The sine quotient in this formula is understood by its smooth
extension at \(q=0\).  Thus
\[
 H(0,\xi,\rho,0)=\log\frac{2\xi}{\rho}.
\]
For every \(\xi\in I\), the limiting root is \(\rho=2\xi\), and
\[
 H_\rho(0,\xi,2\xi,0)=-\frac1{2\xi}.
\]
Because \(I\Subset(0,\infty)\), these derivatives are uniformly
bounded away from zero.  The parameter-dependent implicit-function
theorem, applied along the compact set
\(\{(0,\xi,2\xi,0):\xi\in I\}\), therefore produces constants
\(q_*,\delta_*>0\) and a unique smooth branch
\[
 \rho=\rho(q,\xi,\delta)
\]
for \(|q|\leq q_*\), \(\xi\in I\), and
\(|\delta|\leq\delta_*\).  By decreasing these constants, we may
assume that \(H_\rho\) stays uniformly separated from zero throughout
the corresponding compact parameter box.

The branch just obtained is the scaled version of the root defined in
Step~1.  In fact, \(\rho=2\xi+o(1)\) uniformly on \(I\), so
\(\rho-\xi\) stays positive and bounded away from zero.  Therefore,
for small positive \(q\),
\[
 0<\frac{\pi}{2}-q\rho
 <\frac{\pi}{2}-q\xi<\frac{\pi}{2}.
\]
It follows that the scaled branch corresponds to a root in
\((0,\gamma)\); uniqueness of the root then identifies it with
\(\mathscr R(\gamma,\tau)\).

Define
\[
 \rho_0(q,\xi)=\rho(q,\xi,0),
 \qquad
 \rho_1(q,\xi)=
 \partial_\delta\rho(q,\xi,0).
\]
Taylor's formula in \(\delta\) gives
\begin{equation}\label{eq:rho-expansion}
 \rho(q,\xi,\delta)
 =\rho_0(q,\xi)+\delta\rho_1(q,\xi)
 +O_{C^2_\xi}(\delta^2).
\end{equation}
More precisely,
\begin{equation}\label{eq:rho-remainder}
 \bigl\|\rho(q,\cdot,\delta)
 -\rho_0(q,\cdot)-\delta\rho_1(q,\cdot)
 \bigr\|_{C^2(I)}
 \leq C\delta^2.
\end{equation}
To justify the uniform \(C^2\) estimate, first fix a compact interval
\(I'\Subset(0,\infty)\) with \(I\Subset\operatorname{int}I'\), and
choose a compact interval \(K_\rho\Subset(0,\infty)\) whose interior
contains \(\{2\xi:\xi\in I'\}\).  On
\[
 [-q_*,q_*]\times I'\times K_\rho\times[-\delta_*,\delta_*],
\]
write every sine quotient using
\(\operatorname{sinc}(z)=\sin z/z\).  The decomposition
\eqref{eq:W-smooth-decomposition} then shows that \(H\), including its
extension at \(q=0\), is \(C^\infty\) on this compact set.  In
addition,
\[
 H_\rho(0,\xi,\rho,0)=-\frac1\rho
 \qquad
 (\xi\in I',\ \rho\in K_\rho).
\]
Hence \(|H_\rho|\geq(\max K_\rho)^{-1}\) on the limiting parameter
slice.  By uniform continuity, after decreasing \(q_*,\delta_*\), one
has
\[
 |H_\rho|\geq\frac1{2\max K_\rho}
\]
throughout the displayed parameter box.  The implicit branch remains
in the interior of \(K_\rho\) after the same decrease.  All derivatives
of \(H\) needed below are therefore uniformly bounded there, and the
division by \(H_\rho\) is uniform.

Differentiate the implicit identity
\[
 H(q,\xi,\rho(q,\xi,\delta),\delta)=0.
\]
Solving successively for the derivatives containing the highest
derivative of \(\rho\), first twice in \(\delta\) and then up to twice
in \(\xi\), expresses
\(\partial_\xi^j\partial_\delta^2\rho\), \(0\leq j\leq2\), as finite
sums of products of lower derivatives of \(\rho\) and derivatives of
\(H\), divided by powers of \(H_\rho\).  Induction on the total order
therefore gives
\[
 \sup_{\substack{|q|\leq q_*\\|\delta|\leq\delta_*}}
 \bigl\|
 \partial_\delta^2\rho(q,\cdot,\delta)
 \bigr\|_{C^2(I)}
 \leq C.
\]
The integral form of Taylor's theorem in \(\delta\) now yields
\eqref{eq:rho-remainder}, with \(C\) independent of \(q\) and
\(\delta\).

At \(\delta=0\), equation \eqref{eq:scaled-root} gives
\[
 \sin(q\rho_0)=2\sin(q\xi),
\]
so, on the branch under consideration,
\[
 \rho_0(q,\xi)
 =\frac1q\arcsin\bigl(2\sin(q\xi)\bigr)
 =2\xi+q^2\xi^3+O_{C^2(I)}(q^4).
\]
Differentiating the implicit identity once in \(\delta\) gives
\[
 H_\rho\rho_\delta+H_\delta=0.
\]
At \(\delta=0\),
\[
 H_\rho(q,\xi,\rho_0,0)
 =-q\cot(q\rho_0),
 \qquad
 H_\delta(q,\xi,\rho_0,0)
 =\Wcal(q,\xi,\rho_0),
\]
and hence
\[
 \rho_1(q,\xi)
 =\frac{\Wcal(q,\xi,\rho_0(q,\xi))}
 {q\cot(q\rho_0(q,\xi))}.
\]
In view of \eqref{eq:Wlimit},
\[
 \Wcal(0,\xi,2\xi)=-\frac{\pi}{8\xi^3},
 \qquad
 \rho_1(0,\xi)=-\frac{\pi}{4\xi^2}.
\]
The smooth dependence established above therefore implies
\[
 \rho_1(q,\cdot)
 =-\frac{\pi}{4\xi^2}+o_{C^2(I)}(1),
\]
and in particular
\[
 (\rho_1)_{\xi\xi}(q,\xi)
 =-\frac{3\pi}{2\xi^4}+o(1)
\]
uniformly on \(I\).

For comparison, let
\[
 R_0(\gamma)=\frac{\pi}{2}-q\rho_0(q,\xi),
 \qquad
 \gamma=\frac{\pi}{2}-q\xi.
\]
The uncorrected root satisfies
\[
 \cos R_0(\gamma)=2\cos\gamma.
\]
Differentiating once gives
\[
 (R_0)_\gamma
 =\frac{2\sin\gamma}{\sin R_0}.
\]
Differentiating again and using
\(\cos R_0=2\cos\gamma\) yields
\begin{equation}\label{eq:R0-second}
 (R_0)_{\gamma\gamma}
 =-\frac{6\cos\gamma}
 {(1-4\cos^2\gamma)^{3/2}}
 =-6q\xi+O(q^3),
\end{equation}
uniformly for \(\xi\in I\).

We now set \(\tau=q^{9/2}\), and hence
\(\delta=q^{3/2}\).  Since
\[
 \mathscr R=\frac{\pi}{2}-q\rho,
 \qquad
 \partial_\gamma=-q^{-1}\partial_\xi,
\]
we have the exact identities
\[
 \mathscr R_\gamma=\rho_\xi,
 \qquad
 \mathscr R_{\gamma\gamma}=-q^{-1}\rho_{\xi\xi}.
\]
Using \eqref{eq:rho-expansion},
\eqref{eq:rho-remainder}, and \eqref{eq:R0-second}, we obtain
\begin{align}
 \mathscr R_{\gamma\gamma}
 &=(R_0)_{\gamma\gamma}
 -\frac{\delta}{q}(\rho_1)_{\xi\xi}
 +O\!\left(\frac{\delta^2}{q}\right)\notag\\
 &=-6q\xi+O(q^3)
 +q^{1/2}\left(\frac{3\pi}{2\xi^4}+o(1)\right)
 +O(q^2)\notag\\
 &=q^{1/2}\left(\frac{3\pi}{2\xi^4}+o(1)\right),
 \label{eq:Rgg-computed}
\end{align}
uniformly for \(\xi\in I\).  This proves
\eqref{eq:Rgg-target}.  The same expansion, now used without two
\(\xi\)-derivatives, gives
\begin{equation}\label{eq:Rgamma-computed}
 \mathscr R_\gamma=\rho_\xi=2+O(q^{3/2}),
 \qquad
 \rho=2\xi+O(q^{3/2})
\end{equation}
uniformly on \(I\).

\medskip
\noindent\emph{Step 3: convexity of the composed level graph.}
Define
\[
 \xi_q(x)=\frac{\pi/2-\gamma_q(x)}q.
\]
The definition of \(\gamma_q\) gives the exact formula
\[
 \xi_q(x)
 =\frac{\arcsin q}{q}-\frac{\lambda x}{q}+q^6x^2.
\]
For \(|x|\leq H_q\),
\[
 \left|\frac{\lambda x}{q}\right|\leq\theta,
 \qquad
 q^6x^2\leq Cq^{7/2},
\]
and hence
\[
 \xi_q(x)
 =1-\frac{\lambda x}{q}+O(q^2)+O(q^{7/2}).
\]
Our choice of \(\theta\) therefore ensures, after decreasing \(q_0\),
that \(\xi_q(x)\in I\) for every \(|x|\leq H_q\).

Twice differentiating
\(r_q(x)=\mathscr R(\gamma_q(x),\tau)\) gives
\begin{equation}\label{eq:r-chain-needed}
 r_q''(x)
 =\mathscr R_{\gamma\gamma}(\gamma_q(x),\tau)
 (\gamma_q'(x))^2
 +\mathscr R_\gamma(\gamma_q(x),\tau)\gamma_q''(x).
\end{equation}
On \(|x|\leq H_q\), direct differentiation of \(\gamma_q\) yields
\[
 (\gamma_q')^2=q^{9/2}+O(q^8),
 \qquad
 \gamma_q''=-2q^7.
\]
Combining these estimates with
\eqref{eq:Rgg-computed} and \eqref{eq:Rgamma-computed}, we find
\begin{equation}\label{eq:r-second-proved}
 r_q''(x)
 =q^5\left(
 \frac{3\pi}{2\xi_q(x)^4}+o(1)
 \right)-4q^7+o(q^7),
\end{equation}
where the \(o(1)\) is uniform for \(|x|\leq H_q\).  Since
\(I\Subset(0,\infty)\), the leading coefficient has a positive
uniform lower bound.  It follows that there exist constants
\(c_1,C_1>0\), independent of small \(q\), such that
\begin{equation}\label{eq:r-second-needed}
 c_1q^5\leq r_q''(x)\leq C_1q^5
 \qquad (|x|\leq H_q).
\end{equation}

The localization in \eqref{eq:Rgamma-computed} also gives
\[
 \gamma_q(x)-r_q(x)
 =q\bigl(\rho(q,\xi_q(x),q^{3/2})-\xi_q(x)\bigr)
 \asymp q
\]
uniformly for \(|x|\leq H_q\).  Indeed,
\(\rho-\xi=\xi+O(q^{3/2})\), while \(\xi\in I\).

Since \((V_q)_y>0\) along the upper branch \(y=r_q(x)\), the
sublevel set \(\{V_q<c_*\}\) lies locally below this graph.  Hence
the estimate \(r_q''>0\) already shows that this sublevel set is
locally nonconvex.  Equivalently, with respect to the outward unit
normal
\[
 \nu=\frac{DV_q}{|DV_q|}
 =\frac{(-r_q',1)}{\sqrt{1+(r_q')^2}},
\]
its boundary curvature is
\[
 \kappa_{\mathrm{out}}
 =-\frac{r_q''}{(1+(r_q')^2)^{3/2}}<0.
\]
This differential observation is not stable under the available
\(C^0\)-comparison with \(u_q\).  We therefore extract from it a
finite midpoint defect with a quantitative positive margin.

\medskip
\noindent\emph{Step 4: a finite midpoint defect for \(V_q\).}
Define
\begin{equation}\label{eq:Deltadef}
 \Delta_q
 =\frac{r_q(H_q)+r_q(-H_q)}2-r_q(0).
\end{equation}
Taylor's formula with integral remainder gives
\begin{align*}
 r_q(H_q)-r_q(0)-H_qr_q'(0)
 &=\int_0^{H_q}(H_q-t)r_q''(t)\,\dd t,\\
 r_q(-H_q)-r_q(0)+H_qr_q'(0)
 &=\int_0^{H_q}(H_q-t)r_q''(-t)\,\dd t.
\end{align*}
Adding these identities and dividing by two yields the exact formula
\[
 \Delta_q
 =\frac12\int_0^{H_q}(H_q-t)
 \bigl[r_q''(t)+r_q''(-t)\bigr]\dd t.
\]
Since
\[
 \int_0^{H_q}(H_q-t)\,\dd t=\frac{H_q^2}{2},
\]
the bounds \eqref{eq:r-second-needed} and
\(H_q=\theta q^{-5/4}\) imply
\begin{equation}\label{eq:Deltascale}
 c_3q^{5/2}\leq\Delta_q\leq C_3q^{5/2}
\end{equation}
for constants \(c_3,C_3>0\) independent of \(q\).

Set
\begin{equation}\label{eq:points}
 d_q=\frac{\Delta_q}{4},
 \qquad
 P_\pm=(\pm H_q,r_q(\pm H_q)-d_q),
 \qquad
 M_q=\frac{P_++P_-}{2}.
\end{equation}
By the definition of \(\Delta_q\),
\begin{align*}
 (M_q)_x&=0,\\
 (M_q)_y
 &=\frac{r_q(H_q)+r_q(-H_q)}2-d_q
 =r_q(0)+\Delta_q-d_q
 =r_q(0)+\frac{3\Delta_q}{4}.
\end{align*}

We now check both that these points lie in \(\Om_q\) and that
\((V_q)_y\) has the required size on the intervening vertical
segments.  The preceding localization gives
\[
 \gamma_q(x)-r_q(x)\geq cq
 \qquad (|x|\leq H_q).
\]
Since \(d_q+\Delta_q=O(q^{5/2})=o(q)\), we have, for small \(q\),
\[
 r_q(\pm H_q)-d_q<\gamma_q(\pm H_q)
\]
and
\[
 \gamma_q(0)-(M_q)_y
 =\gamma_q(0)-r_q(0)-\frac{3\Delta_q}{4}
 \geq cq-Cq^{5/2}>0.
\]
Moreover, \(r_q(x)=\pi/2-q\rho\) stays uniformly positive, and the
vertical displacements are \(o(1)\).  Hence none of these points can
meet the lower boundary \(y=-\gamma_q(x)\).
Thus the three points, as well as the short vertical segments joining
them to the level graph, lie in the channel.  Moreover,
\(H_q/L_q=\theta/\Theta<1\), so
\eqref{eq:central-domain} shows that they lie strictly inside
\(\Om_q\).

On all these vertical segments,
\[
  \frac{\pi}{2}-y\asymp q,
 \qquad
 \tan y\asymp q^{-1}.
\]
Using the explicit derivative of the corrector displayed in Step~1,
and also \(\tan\gamma_q\asymp q^{-1}\), we obtain
\[
 |w_y(\gamma_q(x),y)|\leq Cq^{-4}.
\]
Since \(\tau=q^{9/2}\),
\[
 \tau|w_y(\gamma_q(x),y)|\leq Cq^{1/2}.
\]
Therefore
\[
 (V_q)_y
 =\tan y+\tau w_y(\gamma_q(x),y)
\]
satisfies, after decreasing \(q_0\),
\begin{equation}\label{eq:Vy-needed}
 \frac{c_4}{q}\leq(V_q)_y\leq\frac{C_4}{q}
\end{equation}
on every vertical segment used below.

Since \(V_q(\pm H_q,r_q(\pm H_q))=c_*\), the fundamental theorem of
calculus and \eqref{eq:Vy-needed} give
\begin{align}
 V_q(P_\pm)-c_*
 &=-\int_{r_q(\pm H_q)-d_q}^{r_q(\pm H_q)}
 (V_q)_y(\pm H_q,s)\,\dd s
 \leq-\frac{c_4}{q}d_q,
 \notag\\
 V_q(M_q)-c_*
 &=\int_{r_q(0)}^{r_q(0)+3\Delta_q/4}
 (V_q)_y(0,s)\,\dd s
 \geq\frac{3c_4}{4q}\Delta_q.
 \label{eq:model-gap-from-root}
\end{align}
Combining these inequalities with
\(d_q=\Delta_q/4\) and \eqref{eq:Deltascale}, we obtain a constant
\(c_2>0\), independent of \(q\), such that
\begin{equation}\label{eq:model-gap}
 V_q(P_\pm)\leq c_*-c_2q^{3/2},
 \qquad
 V_q(M_q)\geq c_*+c_2q^{3/2}.
\end{equation}

\medskip
\noindent\emph{Step 5: transfer to the exact Dirichlet solution.}
Because
\[
 \frac{H_q}{L_q}=\frac{\theta}{\Theta}<\frac14,
\]
all three points lie in
\(\Om_q\cap\{|x|\leq\eta L_q\}\) with, for instance,
\(\eta=1/2\).  The central comparison estimate
\eqref{eq:central-comparison} gives at these points
\begin{equation}\label{eq:comparison-smaller-gap}
 |u_q-V_q|
 \leq Cq^2+C|\log q|e^{-cq^{-1/4}}
 =o(q^{3/2}).
\end{equation}
After decreasing \(q_0\), the right-hand side is at most
\(\frac12c_2q^{3/2}\).  Applying
Corollary~\ref{cor:gap-transfer} to \eqref{eq:model-gap}, with
\(\delta=c_2q^{3/2}\), yields
\[
 u_q(P_\pm)
 \leq c_*-\frac{c_2}{2}q^{3/2},
 \qquad
 u_q(M_q)
 \geq c_*+\frac{c_2}{2}q^{3/2}.
\]
Thus \eqref{eq:uq-gap} holds with \(c_0=c_2/2\).  In particular,
\(P_+,P_-\in\{u_q<c_*\}\), whereas their midpoint \(M_q\) does not
 belong to this sublevel set.  Hence \(\{u_q< -\log 2\}\) is not convex.

Finally, suppose that \(c\) satisfies \eqref{eq:nearby-levels}.  Then
\[
 u_q(P_\pm)
 \leq c_*-c_0q^{3/2}
 <c_*-\frac{c_0}{2}q^{3/2}<c,
\]
while
\[
 u_q(M_q)
 \geq c_*+c_0q^{3/2}
 >c_*+\frac{c_0}{2}q^{3/2}>c.
\]
The same two endpoints and the same midpoint therefore prove that
\(\{u_q<c\}\) is not convex.  This completes the proof.
\end{proof}

\begin{proof}[Proof of Theorem~\ref{thm:main}]
Choose any sufficiently small \(q\), and set
\(\Om=\Om_q\), \(u=u_q\).  Theorem~\ref{thm:stable-gap} proves that
\(\{u< -\log 2\}\) is nonconvex.  This
proves the assertion of Theorem~\ref{thm:main}.
\end{proof}

\section*{Acknowledgments}
 The  third authors thanks Professor Xu-Jia Wang for bringing this question to his attention in 2002 when he was a postdoctoral fellow in Australia National University. The  authors are supported by National Key R\&D Program of China 2025YFA1017603. Guohuan Qiu was also supported by the National Natural Science Foundation of China (Grant No. 12571227).

\end{document}